\documentclass[Final]{siamltex}
\usepackage[margin=1.2in]{geometry}
\usepackage{amsfonts}
\usepackage{mathrsfs}
\usepackage{amsmath}
\usepackage{amssymb}
\usepackage{amsmath}
\usepackage{float}
\usepackage{booktabs}
\usepackage{multirow}
\usepackage{ifpdf}
\usepackage{enumerate}
\usepackage{graphicx}
\usepackage{color}
\usepackage{lineno}

\usepackage{caption}
\usepackage[justification=centering]{caption}
\usepackage{bm}

\makeatletter

\newcommand{\Rmnum}[1]{\expandafter\@slowromancap\romannumeral #1@}
\makeatother

\newtheorem{assumption}[theorem]{Assumption}
\newtheorem{algorithm}[theorem]{Algorithm}
\newtheorem{propo}[theorem]{Proposition}

\usepackage{xcolor}	 
\definecolor{mygreen}{RGB}{44,85,17}
\definecolor{myblue}{RGB}{34,31,217}
\definecolor{mybrown}{RGB}{194,164,113}
\definecolor{myred}{RGB}{255,66,56}
\definecolor{mypurple}{RGB}{255 250 205}
\definecolor{myjackieblue}{RGB}{11,23,70}
\definecolor{mygrey}{RGB}{230 230 250}

\begin{document}

\title{Multi-grid Multi-Level Monte Carlo Method for Stokes-Darcy interface Model with Random Hydraulic Conductivity\thanks{This work is partially supported by NSF grants DMS-1418624 and DMS-1722647, National Science Foundation of China (91330104).}}

\author{Zhipeng Yang\thanks{Division of Applied and Computational Mathematics, Beijing Computational Science Research Center, Beijing 100094, P. R. China, yangzhp@csrc.ac.cn} \and Ju Ming \thanks{School of Mathematics and Statistics, Huazhong University of Science and Technology, Wuhan 430074, P. R. China, jming@hust.edu.cn, corresponding author.} \and Xiaoming He \thanks{Department of Mathematics and Statistics, Missouri University of Science and Technology, Rolla, MO 65409, U. S. A., hex@mst.edu} \and Li Zhang \thanks{School of mathematical Sciences, University of Electronic Science and Technology, Chengdu Sichuan 611731, P. R. China, lizhang\_137363@163.com} }

\date{}
\maketitle

\begin{abstract}
In this article we develop a multi-grid multi-level Monte Carlo (MGMLMC) method for the stochastic Stokes-Darcy interface model with random hydraulic conductivity both in the porous media domain and on the interface. Because the randomness through the interface affects the flow in the Stokes domain, we investigate the coupled stochastic Stokes-Darcy model to improve the fidelity as this model also considers the second and third porosity of the free flow. Then we prove the existence and uniqueness of the weak solution of the variational form. For the numerical solution, we adopt the Monte Carlo (MC) method and finite element method (FEM), for the discrete form in the probability space and physical space, respectively. In the traditional single-level Monte Carlo (SLMC) method, more accurate numerical approximate requires both larger number of samples in probability space and smaller mesh size in the physical space. Then the computational cost increase significantly, which is the product of the number of samples and the computational cost of each sample, as the mesh size becomes smaller for the more accurate numerical approximate. Therefore we adopt the multi-level Monte Carlo (MLMC) method to dramatically reduce the computational cost in the probability space, because the number of samples decays fast while the mesh size decreases. We also develop a strategy to calculate the number of samples needed in MLMC method for the stochastic Stokes-Darcy model. Furthermore MLMC naturally provides the hierarchial grids and sufficient information on these grids for multi-grid (MG) method, which can in turn improve the efficiency of MLMC. In order to fully make use of the dynamical interaction between this two methods, we propose the multi-grid multi-level Monte Carlo method for more efficiently solving the stochastic model, with additional efforts on the interface. Numerical examples are provided to verify and illustrate the proposed method and the theoretical conclusions.
\end{abstract}

\begin{keywords}
stochastic Stokes-Darcy interface model, multi-level Monte Carlo, multi-grid method.
\end{keywords}

\begin{AMS}
35R60, 65C05, 65M60 76S05.
\end{AMS}

\section{Introduction}

The Stokes-Darcy interface model has attracted significant attention from scientists and engineers due to its wide range of applications, such as interaction between surface and subsurface flows \cite{ACesmelioglu_BRiviere_2,MDiscacciati_1,MDiscacciati_EMiglio_AQuarteroni_1,RHoppe_PPorta_YVassilevski_1, WLayton_HTran_CTrenchea_1}, industrial filtrations \cite{VJErvin_EWJenkins_SSun_1,NHanspal_AWaghode_VNassehi_RWakeman_1}, groundwater system in karst aquifers \cite{YCao_MGunzburger_FHua_XWang_1, YGao_XMHe_LMei_XYang_1, DHan_DSun_XWang_1, DHan_XWang_Hwu_1}, petroleum extraction \cite{TArbogast_DSBrunson_1,TArbogast_HLLehr_1, JHou_XMHe_CGuo_MWei_BBai_1}, and many others \cite{ACesmelioglu_BRiviere_3, JChen_SSun_XWang_1, CDAngelo_PZunino_1,AEDiegel_XFeng_SMWise_1, HRui_JZhang_1, SKFStoter_PMuller_LCicalese_MTuveri_DSchillinger_TJRHughes_1, DVassilev_IYotov_1}. Therefore it is not surprising that many different numerical methods have been proposed and analyzed for the Stokes-Darcy model, including domain decomposition methods \cite{YBoubendir_STlupova_2,YCao_MGunzburger_XMHe_XWang_2,WChen_MGunzburger_FHua_XWang_1,MDiscacciati_AQuarteroni_1, MDiscacciati_AQuarteroni_AValli_1, MGunzburger_XMHe_BLi_1, DVassilev_CWang_IYotov_1}, Lagrange multiplier methods \cite{IBabuska_GNGatica_1,GNGatica_SMeddahi_ROyarzua_1,GNGatica_ROyarzua_FJSayas_2,PHuang_JChen_MCai_1,WJLayton_FSchieweck_IYotov_1}, discontinuous Galerkin methods \cite{PChidyagwai_BRiviere_1, VGirault_BRiviere_1,GKanschat_BRiviere_1,KLipnikov_DVassilev_IYotov_1,BRiviere_1,BRiviere_IYotov_1}, multi-grid methods \cite{TArbogast_MGomez_1,MCai_MMu_JXu_2,MMu_JXu_1}, partitioned time stepping methods \cite{MKubacki_MMoraiti_1,MMu_XZhu_1, LShan_HZheng_1}, coupled finite element methods \cite{JCamano_GNGatica_ROyarzua_RRuizBaier_PVenegas_1, TKarper_KAMardal_RWinther_1,AMarquez_SMeddahi_FJSayas_2, HRui_RZhang_1}, and many others \cite{LBadea_MDiscacciati_AQuarteroni_1, YBoubendir_STlupova_1,WChen_MGunzburger_DSun_XWang_1, VJErvin_EWJenkins_HLee_1, JGalvis_MSarkis_1,VGirault_DVassilev_IYotov_1, PHuang_JChen_MCai_1, SMunzenmaier_GStarke_1, IRybak_JMagiera_1, STlupova_RCortez_1, WWang_CXu_1}.

The above existing works only consider the deterministic Stokes-Darcy model, for which the problem data, including the model coefficients, the forcing terms, the domain geometry, the boundary conditions and the initial conditions, are assumed to be perfectly known. However, in reality there is a significant amount of uncertainty involved in determining these real-life data due to measurements and simplifications \cite{JHCushman_porous_media_2013, RGhanem_porous_1998, HHNajm_UQ_flow_2009, UQ_Smith_2013}. \textcolor{red}{}

There are some works on the uncertainties of the porous media flow by assuming the hydraulic conductivity of the porous media is a random field in the second order elliptic equation \cite{TYHou_stochastic_porous_multiscale_2008, BGanis_stochastic_porous_media_2008, DZhang_porous_2007,DZhang_porous_2004}. But the Stokes-Darcy model has a much more complicated system for the uncertainties due to the flow interaction on the interface between the porous media flow and the free flow in conduits. Hence it is not trivial to study the effect of randomness of the hydraulic conductivity on the whole coupled flow performance, which is key component of this paper, especially around the interface.

On the other hand, in the numerical simulation area, the Monte Carlo method \cite{MC_Robert} has been a widely applied to solve the stochastic problems. The convergence of the Monte Carlo method is based on the number of the samplings. Unfortunately, for a high accuracy result, one usually needs a large number of samples, which significantly increases the computational cost. To develop an accurate and efficient numerical method for simulating the coupled stochastic porous media flow and free flow, we develop a multi-level Monte Carlo method \cite{Schwab_MLMC_2011, Teckentrup_2013, Giles_2008,  Giles_2015,  Schwab_MLMC_2015, Giles_2013} to solve the sophisticated stochastic Stokes-Darcy interface model. This method is much more costly efficient by significantly reducing the number of samples on the fine meshes. But it is not trivial to determine how many samples should be used in each level to keep the global accuracy while minimizing the cost. Therefore, we develop a strategy based on a detailed analysis to overcome this difficulty.

Furthermore, the multi-level Monte Carlo method only reduces the computational cost in the probability space, not in the physical space. Inspired by a fact that the multi-level Monte Carlo method already has a set of hierarchical grids for the multi-level idea, it is a natural idea to fully make use of the same set of hierarchical grids to solve the discrete algebraic system by using the powerful multi-grid method \cite{Briggs_MG_2000, OLM_2003,  Motsa_2015, MWang_LChen_MG_2013, LChen_MG_2015}, which can further improve the efficiency of the proposed multi-level Monte Carlo method. Meanwhile, the saved information of the multi-level Monte Carlo method on the set of hierarchical grids will also significantly reduce the computational cost of the multi-grid method. Therefore, we combine the multi-level Monte Carlo method and the multi-grid method on the same set of hierarchical grids to propose an even more costly efficient method, which is the multi-grid multi-level Monte Carlo method.

The rest of the paper is organized as follows. In section 2, we briefly recall the deterministic Stokes-Darcy model. In section 3, we present the stochastic Stokes-Darcy interface model, the weak formulation of the stochastic Stokes-Darcy model and the proof of the well-posedness. In section 4, we recall the Monte Carlo method to approximate the numerical moments of the stochastic solutions, adopt the multi-level Monte Carlo method to reduce the computational cost in probability space, and then develop the multi-grid multi-level Monte Carlo to further reduce the computational cost. In section 5 we provide numerical examples to verify the theoretical analysis and illustrate the features of the proposed methods.

\section{Deterministic model for coupled fluid flow with porous media flow}

The coupled Stokes-Darcy system describes the the free flow by Stokes equations in the conduit domain and the confined flow by Darcy system in the porous media domain. And three interface conditions displaced follow are used to couple the flows in these two domains. In this paper, we consider the coupled Stokes-Darcy system on a bounded domain $ D_{ms} = D_m \cup D_s \subset {\mathbb R}^{d}$, $d = 2, 3$, where $D_m$ is the porous media domain and $D_s$ is the conduit domain. We decompose the boundary $\partial D$ into two parts: $\Gamma_m=\partial D_m \backslash \Gamma_I$, $\Gamma_s=\partial D_s \backslash \Gamma_I$, and denote the interface as $\Gamma_I=\partial D_m\cap\partial D_s$.

In the porous media domain $D_m$, the flow is governed by the Darcy system \cite{GIBarenblatt_IPZheltov_INKochina_1}
\begin{eqnarray}
    \vec{u}_m(x)&=&-{\mathbb{K}(x)}\nabla \phi_m(x) \hspace{.4cm} \text{in}\ D_m,  \label{Darcy_law_1}\\
    \nabla \cdot \vec{u}_m(x)&=&f_m(x) \hspace{.4cm} \text{in}\ D_m, \label{Darcy_law_2}
\end{eqnarray}
\noindent here, $\vec{u}_m$ denotes the specific discharge in the porous media, $\mathbb{K}$ is the hydraulic conductivity tensor of the porous media that is symmetric and positive definite in accordance with physical meaning, $\phi_m$ is the hydraulic head, and $f_m$ is the sink/source term.

\par By substituting \eqref{Darcy_law_1} into \eqref{Darcy_law_2}, we obtain the second-order form of the Darcy system
\begin{equation}
    -\nabla \cdot (\mathbb{K}(x)\nabla\phi_m(x))=f_m(x) \hspace{.4cm} \text{in}\ D_m. \label{Darcy_law_4}
\end{equation}

\par In the conduit domain $D_s$, the flow is governed by the Stokes equations:
\begin{eqnarray}
    -\nabla \cdot \mathbb{T}(\vec{u}_s,p_s)&=&\vec{f}_s \hspace{.4cm} \text{in}\ D_s,
    \label{Stokes_1} \\
    \nabla \cdot \vec{u}_s&=&0 \hspace{.4cm} \text{in}\ D_s, \label{Stokes_2}
\end{eqnarray}
\noindent where $\vec{u}_s$ denotes the fluid velocity, $p_s$ is the kinematic pressure, and $\vec{f_s}$ is the external body force. $\mathbb{T}$ is the stress tensor, defined as $\mathbb{T}(\vec{u}_s,p_s)=2\nu\mathbb{D}(\vec{u}_s)-p_s\mathbb{I}$, where $\nu$ is the kinematic viscosity of the fluid and $\mathbb{D}(\vec{u}_s)=\frac{1}{2}(\nabla\vec{u}_s+\left({\nabla}\vec{u}_s\right)^{T})$.

On the interface between the conduit and the porous media domain, we impose three interface conditions:
\begin{eqnarray}
    \vec{u}_s\cdot\vec{n}_s&=&(\mathbb{K}\nabla\phi_m)\cdot\vec{n}_m  \hspace{.4cm} \text{on}\ \Gamma_I, \label{BJ_1} \\
    -\vec{n}_s^T \mathbb{T}(\vec{u}_s,p_s)\vec{n}_s&=&g(\phi_m-z) \hspace{.4cm} \text{on}\ \Gamma_I, \label{BJ_2} \\
    -\bm{\tau}_j^T \mathbb{T}(\vec{u}_s,p_s)\vec{n}_s&=&\frac{\alpha\nu\sqrt{\mathbf{d}}}
    {\sqrt{\text{trace}(\Pi(x))}}\bm{\tau}_j^T (\vec{u}_s+\mathbb{K}\nabla\phi_m) \hspace{.4cm} \text{on}\ \Gamma_I, \label{BJ_3}
\end{eqnarray}
\noindent where $\vec{n}_s$, $\vec{n}_m$ denote the unit outer normal to the conduit and the porous media regions at the interface $\Gamma_I$, respectively, \boldmath$\tau$\unboldmath$_j(j=1,...,d-1)$ denote mutually orthogonal unit tangential vectors to the interface $\Gamma_I$, $z$ is the hight, $g$ is the gravitational acceleration, and $\Pi(x)=\frac{\mathbb{K}(x)\nu}{g}$ is the intrinsic permeability. The first interface condition \eqref{BJ_1} is governed by the conservation of mass, the second interface condition \eqref{BJ_2} represents the balance of the kinematic pressure in the matrix and the stress in the free flow at the normal direction along the interface, and the last interface condition \eqref{BJ_3} is the famous Beavers-Joseph condition \cite{GBeavers_DJoseph_1, YCao_MGunzburger_XMHe_XWang_1, YCao_MGunzburger_XHu_FHua_XWang_WZhao_1, YCao_MGunzburger_FHua_XWang_1, WFeng_XMHe_ZWang_XZhang_1, XMHe_JLi_YLin_JMing_1}.

\section{Stokes-Darcy interface model with random permeability}
To overcome the difficulty of measuring the exact permeability at every point in the porous media domain, we use an underlying random field to describe the intrinsic permeability tensor $\Pi$. Thus the hydraulic conductivity tensor $\mathbb{K}(x)$ is also a random field with the relationship $\Pi=\frac{\mathbb{K}\nu}{g}$. Then we obtain the stochastic partial differential equations to describe the coupled system with the random hydraulic conductivity, based on the deterministic model in the above section. We investigate the uncertainty in the porous domain and the uncertainty transferred to the conduit domain through the interface. Furthermore, we provide the weak formulation and prove the well-posedness of the weak solution of the coupled stochastic model.

\subsection{Functional spaces and notations}
Before the study of the stochastic coupled problem, we introduce some notations. Throughout this paper, we adopt the notations in \cite{Evans_PDE} for the classical Sobolev spaces. Let $D$ be an open, connected, bounded, and convex subset of $\mathbb{R}^d$, $d=2,3$, with polygonal and Lipschitz continuous boundary $\partial D$. Let $r\in \mathbb{R}$, $q \in \mathbb{Z}$, and $W^{r,q}(D)$ be a Sobolev space on $D$ with the standard norm $\|\cdot\|_{W^{r,q}(D)}$ and semi-norm $\mid\cdot\mid_{W^{r,q}(D)}$.

Let $(\Omega, \mathcal{F}, \mathcal{P})$ be a complete probability space. Here $\Omega$ is the set of outcomes, $\mathcal{F}$ is the $\sigma$-algebra of events, and $\mathcal{P}:\mathcal{F}\rightarrow[0,1]$ is a probability measure.

For the given probability space $(\Omega,\mathcal{F},\mathcal{P})$ and the Sobolev space $W^{r,q}(D)$ with the inner product $(\cdot,\cdot)_{W^{r,q}(D)}$ and norm $\|\cdot\|_{W^{r,q}(D)}$, we define the stochastic Sobolev space, which consists of strongly measurable, $r$-summable mappings $\phi:\Omega\rightarrow W^{r,q}(D)$, by
\begin{equation*}
    L^2\left(\Omega;W^{r,q}(D)\right):=\{\phi:\Omega\rightarrow W^{r,q}(D)\ |\ \phi \ \text{strongly measurable}, \ \|\phi\|_{L^2\left(\Omega;W^{r,q}(D)\right)}<\infty \}.
\end{equation*}
Here $\|\cdot\|_{L^2\left(\Omega;H^r(D)\right)}$ is the norm given as, $\forall\phi\in L^2\left(\Omega;W^{r,q}(D)\right)$,
\begin{equation*}
    \|\phi\|_{L^2\left(\Omega;W^{r,q}(D)\right)}:=
    \left(\mathbb{E}\left[\|\phi(\omega,\cdot)\|^2_{W^{r,q}(D)}\right]\right)^{1/2}:=
    \left(\int_{\Omega}\|\phi(\omega,\cdot)\|^2_{W^{r,q}(D)}d\mathcal{P}(\omega)\right)^{1/2},
\end{equation*}
which is induced by following inner product, $\forall\phi,\psi\in L^2(\Omega;W^{r,q}(D))$,
\begin{equation*}
    [\phi,\psi]_{L^2(\Omega;W^{r,q}(D))}:=
    \mathbb{E}\left[(\phi,\psi)_{W^{r,q}(D)}\right]:=
    \int_{\Omega}(\phi,\psi)_{W^{r,q}(D)}d\mathcal{P}(\omega).
\end{equation*}

For $q=2$, we denote the Hilbert space $H^r(D):=W^{r,2}(D)$ and $H^r_0(D):=\{u:u\in H^{r}(D),u\mid_{\partial D}=0\}$ with the standard norm $\|\cdot\|_{H^{r}(D)}$ and semi-norm $\mid\cdot\mid_{H^{r}(D)}$. For $r=2$, we denote $L^{q}(D):=W^{2,q}(D)$ with the standard norm $\|\cdot\|_{L^{2}(D)}$. For $d=2,3$, we denote $\mathbf{H}^r(D):=\left[H^r(D)\right]^d$ and $\mathbf{L}^{q}(D):=\left[L^{q}(D)\right]^d$. For the vector $\vec{v}=(v_1, v_2, \cdots, v_n)^{\top}$, $n \in \mathbb{N}^+$,  2-norm $\| \vec{v} \|_2$ of $\vec{v}$ is $ \| \vec{v} \|_{2} =  \left( v_1^2 + v_2^2 + \cdots + v_n^2 \right)^{1/2}. $

For simplicity, we define
\begin{eqnarray*}
\mathcal{L}^{q}(D) &=& L^2(\Omega; L^q(D)), \text{\ with norm \ }  \|\cdot\|_{\mathcal{L}(D)} =  \|\cdot\|_{ L^2(\Omega; L^{q}(D)) }, \\
\mathcal{H}^r(D) &=& L^2(\Omega; H^r(D)), \text{\ with norm \ }  \|\cdot\|_{\mathcal{H}^r(D)} =  \|\cdot\|_{ L^2(\Omega; H^r(D)) }, \\
\vec{\mathcal{H}}^r(D) &=& L^2(\Omega; \mathbf{H}^r(D)), \text{\ with norm \ }  \|\cdot\|_{\vec{\mathcal{H}}^r(D)} =  \|\cdot\|_{ L^2(\Omega; \mathbf{H}^r(D)) }.
\end{eqnarray*}

\subsection{Stochastic Stokes-Darcy interface equations}
With the complete probability space $(\Omega, \mathcal{F}, \mathcal{P})$, let $\mathbb{K}(\omega,x)$, $\omega\in\Omega$, $x\in D_m$ be a random hydraulic conductivity tensor.

Then in the porous media domain, the stochastic second-order form of Darcy equation with sink/source term $f_m(x)$ is given as:
\begin{equation}
    -\nabla \cdot \big(\mathbb{K}(\omega,x)\nabla\phi_m(\omega,x) \big)=f_m(x), \hspace{.4cm} \text{in } D_m. \label{Darcy_law_4_stochastic}
\end{equation}

And the interface conditions are modified as:
\begin{eqnarray}
\vec{u}_s(\omega,x)\cdot\vec{n}_s(x)&=&\big(\mathbb{K}(\omega,x)\nabla\phi_m(\omega,x) \big)\cdot\vec{n}_m(x), \hspace{.4cm} \text{on}\ \Gamma_I, \label{stochastic_stokes_Darcy_5} \\
-\vec{n}_s^{\top} \mathbb{T}(\vec{u}_s,p_s) \vec{n}_s&=&g(\phi_m(\omega,x)-z), \hspace{.4cm} \text{on}\ \Gamma_I, \label{stochastic_stokes_Darcy_6} \\
-\bm{\tau}_j^{\top} \mathbb{T}(\vec{u}_s,p_s) \vec{n}_s&=&\frac{\alpha\nu\sqrt{\mathbf{d}}}
{\sqrt{\text{trace}(\Pi(\omega,x))}}\bm{\tau}_j^{\top} \big(\vec{u}_s(\omega,x) + \mathbb{K}(\omega,x)\nabla\phi_m(\omega,x) \big), \hspace{.4cm} \text{on}\ \Gamma_I. \label{stochastic_stokes_Darcy_7}
\end{eqnarray}

Due to the randomness transferred from porous media domain through the interface conditions, the Stokes equations in the conduit domain become stochastic and are given as follows \begin{eqnarray}
-\nabla \cdot \mathbb{T}(\vec{u}_s(\omega,x),p_s(\omega,x))&=&\vec{f}_{s}(x), \hspace{.4cm} \text{in}\ D_s,
\label{stokes_1_stochastic} \\
\nabla \cdot \vec{u}_s(\omega,x)&=&0, \hspace{.4cm} \text{in}\ D_s. \label{stokes_2_stochastic}
\end{eqnarray}

For the boundary conditions, we assume the hydraulic head $\phi_m$ and the fluid velocity $\vec{u}_s$ satisfy homogeneous Dirichlet boundary condition except on $\Gamma_I$.

\subsection{Weak formulation of the coupled problem}
We denote the velocity-pressure spaces on the conduit domain as
\begin{eqnarray*}
X_s &=& \{\vec{u}_s \in \vec{\mathcal{H}}^1(D_s)\ |\  \vec{u}_s=0 \ \text{on} \ \Gamma_s \},\\
X^{0}_s &=& \{\vec{u}_s \in \vec{\mathcal{H}}^{0}(D_s)\ |\  \vec{u}_s=0 \ \text{on} \ \Gamma_s \}, \\
X_{s,div} &=& \{\vec{u}_s \in X_s\ |\ \nabla \cdot \vec{u}_s=0 \ \text{in} \ D_s \},\\
Q_s &=& \{q_s \in \mathcal{L}^2(D_s) \},
\end{eqnarray*}
and we denote the pressure space on the porous media as
\begin{equation*}
X_m=\{\phi_m \in \mathcal{H}^1(D_m)\ |\  \phi_m=0 \ \text{on} \ \Gamma_m \}, \\
X^{0}_m=\{\phi_m \in \mathcal{H}^{0}(D_m)\ |\  \phi_m=0 \ \text{on} \ \Gamma_m \}.
\end{equation*}

For convenience, let $X^{1} = X =X_s\times X_m$, $X_{div}=X_{s,div}\times X_m$, $X^{0}=X^{0}_s\times X^{0}_m$, and $\underline{u}=(\vec{u}_s,\phi_m)\in X$, where $\vec{u}_s\in X_s$, $\phi_m\in X_m$. The norms of $X^{r}, r = 0, 1$ are given as
\begin{eqnarray}
\| \underline{u} \|_{X^{r}} =  \left( \mathbb{E} \left[  \| \underline{u} \|^2_{\mathbf{H}^r(D_s)\times H^r(D_m)}  \right] \right)^{1/2} =  \left( \|\vec{u}_s \|^2_{\vec{\mathcal{H}}^r(D_s)} + \| \phi_m \|^2_{\mathcal{H}^r(D_m)} \right)^{1/2}, \hspace{.4cm} r = 0,\ 1.
\label{norm_u_phi}
\end{eqnarray}

The projection onto the local tangential plane of the vector $\vec{u}$ is denoted as $P_{\tau}(\vec{u})=\vec{u}-(\vec{u}\cdot\vec{n}_s)\vec{n}_s$. Then using the boundary conditions \eqref{stochastic_stokes_Darcy_5}-\eqref{stochastic_stokes_Darcy_7}, we obtain the following weak formulation: find $(\underline{u},p_s)\in X \times Q_s$, such that
\begin{equation}
\left\{
\begin{aligned}
A(\underline{u},\underline{v})-B(\underline{v},p_s)&=F(\underline{v}), \hspace{.4cm} \forall \underline{v}=(\vec{v}_s,\psi_m)\in X, \\
B(\underline{u},q_s)&=0, \hspace{.4cm} \forall q_s \in Q_s,
\end{aligned}
\right.
\label{weak_formulation}
\end{equation}
where
\begin{eqnarray}
A(\underline{u},\underline{v})&=&\mathbb{E}\left[ a(\underline{u},\underline{v}) \right] = \int_{\Omega} a(\underline{u},\underline{v}) d\omega,\\
a(\underline{u},\underline{v})
&=&\int_{D_s}2\nu\mathbb{D}(\vec{u}_s):\mathbb{D}(\vec{v}_s)dx
+ g\int_{D_m}( \mathbb{K}\nabla\phi_m )\cdot \nabla \psi_mdx\\
&+&g\int_{\Gamma_I} \phi_m\vec{v}_s\cdot\vec{n}_sd\Gamma_I
+\int_{\Gamma_I} \frac{\alpha\nu\sqrt{d}}{\sqrt{\text{trace}(\Pi)}}P_{\tau}(\vec{u}_s)\cdot\vec{v}_sd\Gamma_I\\
&-&g\int_{\Gamma_I} (\vec{u}_s\cdot\vec{n}_s) \psi_m d\Gamma_I
+\int_{\Gamma_I} \frac{\alpha\nu\sqrt{d}}{\sqrt{\text{trace}(\Pi)}}P_{\tau}(\mathbb{K}\nabla\phi_m)\cdot\vec{v}_sd\Gamma_I,\\
B(\underline{v},p_s)&=&\mathbb{E}\left[ b(\underline{v},p_s) \right] = \int_{\Omega} b(\underline{v},p_s) d\omega,\\
b(\underline{v},p_s)&=&\int_{D_s}p_s\nabla\cdot\vec{v}_s dx,\\
F(\underline{v})&=&\mathbb{E}\left[ f(\underline{v}) \right] = \int_{\Omega} f(\underline{v}) d\omega, \\
f(\underline{v})
&=&\int_{D_s}\vec{f}_s\cdot\vec{v}_sdx
+g\int_{D_m}f_m\psi_mdx
+\int_{\Gamma_I} gz\vec{v}_s\cdot\vec{n}_sd\Gamma_I.
\end{eqnarray}

\subsection{Well-posedness of the weak solution}
The approach to analyze the well-posedness in our paper is inspired by the ideas in \cite{Schwab_MLMC_2011, YCao_MGunzburger_FHua_XWang_1, GiraultV_RaviartPA_FEM_NS_1986, AQuarteroni_AValli_NumPDE_1994}. One of the following two assumptions is needed to ensure the existence and uniqueness of the weak solution.
\begin{assumption}
Let $\mathbb{K}(\omega, x)$ be a diagonal matrix as $diag\big(K_{11}(\omega, x), \cdots, K_{dd}(\omega, x)\big),  \omega\in\Omega, x\in D_{m}, d=2,3$.
\begin{itemize}
  \item {
    the strong elliptic condition: there are positive lower and upper bounds $K_{min}$, $K_{max}$ such that
     \begin{equation}
    0<K_{min}
    \leq \big\{ K_{ii}(\omega,x) \big\}^d_{i=1}
    \leq K_{max}
    <\infty, \text{\ \ for\ \ } (\omega, x)\in \Omega\times \bar{D}_{m};
    \label{K_assumption_1}
    \end{equation}
   }
  \item {
    the integrability condition: let $K_{min}(\omega):=\min\limits_{x\in \bar{D}_m}\big\{ K_{ii}(\omega,x) \big\}^d_{i=1}$ and $K_{max}(\omega):=\max\limits_{x\in \bar{D}_m}\big\{ K_{ii} \\ (\omega,x) \big\}^d_{i=1}$ satisfy
    \begin{equation}
    0<K_{min}(\omega) \text{\ \ and\ \ } \frac{1}{K_{min}(\omega)},\ K_{max}(\omega)\in L^{\infty}(\Omega) \text{\ \ for \ a.e.\ } \omega\in\Omega.
    \label{K_assumption_2}
    \end{equation}
    }
  \end{itemize}
\end{assumption}
Under the above two assumptions, we derive some properties of the weak formulation.

\begin{lemma}\label{lemma_A_1}
Under the Assumption \eqref{K_assumption_1} or \eqref{K_assumption_2}, the bilinear form $A(\cdot,\cdot)$ is continuous on $X_{div}\times X_{div}$.
\end{lemma}
\begin{proof}
By using the Cauchy-Schwarz inequality, trace theorem and the Assumption \eqref{K_assumption_1} or \eqref{K_assumption_2}, we have
\begin{align*}
    A(\underline{u},\underline{v})
    &\leq 2\nu \| \vec{u}_s \|_{\vec{\mathcal{H}}^1(D_s) )} \| \vec{v}_s \|_{\vec{\mathcal{H}}^1(D_s) )}
    +  g d K_{max} \| \phi_m \|_{\mathcal{H}^1(D_m)}  \| \psi_m \|_{\mathcal{H}^1(D_m)} \\
    & + g \| \phi_m \|_{\mathcal{H}^1(D_m)} \| \vec{v}_s \|_{\vec{\mathcal{H}}^1(D_s)}
    + g \| \psi_m \|_{\mathcal{H}^1(D_m)} \| \vec{u}_s \|_{\vec{\mathcal{H}}^1(D_s)} \\
    & + \frac{\alpha \sqrt{g\nu}}{\sqrt{K_{min}}} \| \vec{u}_s \|_{\vec{\mathcal{H}}^1(D_s)} \| \vec{v}_s \|_{L^2(\vec{\mathcal{H}}^1(D_s)}
    + \frac{\alpha d K_{max}\sqrt{g\nu}}{\sqrt{K_{min}}} \| \phi_m \|_{\mathcal{H}^{1}(D_m)} \| \vec{v}_s \|_{\vec{\mathcal{H}}^1(D_s)},
\end{align*}
for $\forall \underline{u}, \underline{v}\in X_{div}$. Thus the bilinear form $A(\cdot,\cdot)$ is continuous on the space $X_{div}\times X_{div}$.
\end{proof}

\begin{lemma}\label{lemma_F}
The linear form $F(\cdot)$ is continuous on $X_{div}$.
\end{lemma}
\begin{proof}
By using the Cauchy-Schwarz inequality and trace theorem, we have
\begin{align*}
    F(\underline{v})
    \leq \| \vec{f}_s \|_{\vec{\mathcal{H}}^1(D_s)} \| \vec{v}_s \|_{\vec{\mathcal{H}}^1(D_s)}
    + g \| f_m \|_{\mathcal{L}^2(D_m)} \| \psi_m \|_{\mathcal{H}^1(D_m)}
    + gz  \| \vec{v}_s \|_{\vec{\mathcal{H}}^1(D_s)},
\end{align*}
for $\forall \underline{v}\in X_{div}$. Thus the linear form $F(\cdot)$ is continuous on $X_{div}$.
\end{proof}

\begin{lemma}\label{lemma_A_2}
Under the Assumption \eqref{K_assumption_1} or \eqref{K_assumption_2}, the bilinear form $A(\cdot,\cdot)$ is coercive on $X_{div}\times X_{div}$ when the coefficient $\alpha$ in the Beavers-Joseph condition \eqref{stochastic_stokes_Darcy_7} is small enough.
\end{lemma}
\begin{proof}
By using the Korn's inequality, Poincar$\acute{\text{e}}$ inequality, Cauchy-Schwarz inequality, trace theorem and the Assumption \eqref{K_assumption_1} or \eqref{K_assumption_2}, we have
\begin{align*}
    A(\underline{u},\underline{u})
    =&\int_{\Omega} \int_{D_s} 2\nu \mathbb{D}(\vec{u}_s):\mathbb{D}(\vec{u}_s)d D_s d \Omega
    + g \int_{\Omega} \int_{D_m}\big( \mathbb{K}\nabla\phi_m \big)\cdot\big( \nabla \phi_m \big)d D_m d \Omega \\
    +& \int_{\Omega} \int_{\Gamma_I} \frac{\alpha\nu\sqrt{d}}{\sqrt{\text{trace}(\Pi)}} \left( P_{\tau}(\vec{u}_s)\cdot\vec{u}_s + P_{\tau}(\mathbb{K}\nabla\phi_m)\cdot\vec{u}_s \right)d \Gamma_I d \Omega \\
    &\geq  2 C_1 \nu \| \vec{u}_s \|^2_{\vec{\mathcal{H}}^1(D_s)}
    + C_2 g K_{min} \| \phi_m \|^2_{\mathcal{H}^1(D_m)}
    - \frac{\alpha d K_{max}\sqrt{g\nu}}{\sqrt{K_{min}}} \| \phi_m \|_{\mathcal{H}^{1}(D_m)} \| \vec{u}_s \|_{\vec{\mathcal{H}}^1(D_s)}\\
    &\geq C_1 \nu \| \vec{u}_s \|^2_{\vec{\mathcal{H}}^1(D_s)}
    + \frac{1}{2} C_2 g K_{min} \| \phi_m \|^2_{\mathcal{H}^1(D_m)},
\end{align*}
where $ \alpha^2 \leq \frac{2C_1C_2K^2_{min}}{d^2K^2_{max}}$, for $\forall\underline{u}\in X_{div}$. Thus the bilinear form $A(\cdot,\cdot)$ is coercive on $X_{div}\times X_{div}$ when the coefficient $\alpha$ in the Beavers-Joseph \eqref{stochastic_stokes_Darcy_7} condition is small enough.
\end{proof}

\begin{theorem}\label{theorem_well_posed}
Under the Assumption \eqref{K_assumption_1} or \eqref{K_assumption_2}, there exists a unique weak solution $\underline{u}=(\vec{u}_s,\phi_m)\in X$ and $p_{s}$ up to an additive constant for the weak formulation \eqref{weak_formulation} of stochastic Stoke-Darcy interface problem \eqref{Darcy_law_4_stochastic}-\eqref{stokes_2_stochastic} when the coefficient $\alpha$ in the Beavers-Joseph \eqref{stochastic_stokes_Darcy_7} condition is small enough.
\end{theorem}
\begin{proof}
Based on the Lemma \ref{lemma_A_1}, Lemma \ref{lemma_F} and Lemma \ref{lemma_A_2}, there exists a unique weak solution $\underline{u}$ by the Lax-Milgram Lemma. Then the assertion about $p_{s}$ is clear, which drives form the conclusions in the deterministic scenario \cite{GiraultV_RaviartPA_FEM_NS_1986, WJLayton_FSchieweck_IYotov_1, AQuarteroni_AValli_NumPDE_1994}.
\end{proof}

\section{Numerical solution for the stochastic coupled problem}
Since the moments are the characteristic functions of the stochastic solution, the object is to design a numerical method to calculate the moments of the stochastic solution. The main difficulty in this design is how to represent the stochastic solution by a discrete form in the probability space and the physical space. For the discrete form in the probability space, we choose the ensemble representations in sampling methods, e.g., Monte Carlo (MC) method in this paper. But the total computational cost of the traditional single-level Monte Carlo (SLMC) method is very high. Then the multi-level Monte Carlo (MLMC) method is adopted to reduce the total computational cost in the probability space. For the discrete form in the physical space, the finite element method (FEM) is chosen. Furthermore the multi-grid (MG) method is used to reduce the computational cost in physical space. Thus the multi-grid multi-level Monte Carlo (MGMLMC) method is developed to reduce the computational cost both in the probability space and the physical space.

\subsection{Realizations of the random hydraulic conductivity}
The realizations of the random hydraulic conductivity $\mathbb{K}(\omega,x)$ in a discrete form on the spatial domain $D_m$ and the random fields $\Omega$ are the basises of the numerical method. We adopt the grid based method in \cite{IGGraham_GridBased_2011}, because this method represents the random field exactly at the discrete points $x_{1},\cdots , x_{M} \in D_{m}$ without any truncation.

For simplification, we assume $\mathbb{K}(\omega, x) = diag\big( K_{11}(\omega, x), \cdots, K_{dd}(\omega, x)\big), \omega\in\Omega, x\in D_{m}, d=2,3$ is a diagonal matrix. The process to generate the realizations of $K(\omega, x) = K_{11}(\omega, x)$ is displayed as follows, which is as same as the processes to generate the realizations of $K_{ii}(\omega, x), i=2,3$.

Because $K(\omega, x)$ is physical positive, we assume $K(\omega, x)$ is a log-normal distribution, i.e.,
\begin{equation}
K(\omega,x)=e^{Z(\omega,x)}, \ \ \omega\in\Omega, \ x \in\bar{D}_{m}, \label{log_normal_K}
\end{equation}
where $Z(\omega,x)$ is a mean zero Gaussian random field on $\bar{D}_m$, with the continuous covariance function $r(x,y)$, $x,y\in \bar{D}_m$, i.e.,
\begin{eqnarray}
\mathbb{E}[Z(\omega,x)]&=&0, \hspace{.4cm}  \forall x\in \bar{D}_m, \label{mean_Z} \\
\mathbb{E}[Z(\omega,x),Z(\omega,y)]&=&r(x,y), \hspace{.4cm}  \forall x,y\in \bar{D}_m. \label{variance_Z}
\end{eqnarray}

For $x_i \in \bar{D}_m$, $i = 1,2, \cdots, M$, the vector $\vec{x}=(x_1, x_2, \cdots, x_M )^{\top}$ represents all the discrete spatial points in $\bar{D}_m$, on which $Z(\omega,x)$ is provided as $Z(\omega,\vec{x})=(Z_1, Z_2, \cdots, Z_M)^{\top}$, $ Z_{i} = Z(\omega,x_{i}) $. By the covariance function \eqref{variance_Z}, a $M \times M$ positive definite matrix $R$ is given
\begin{equation}
R = \mathbb{E}\big[ Z(\omega,\vec{x}),Z(\omega,\vec{x})^\top \big ]=\big( r(x_i,x_j) \big )^{M}_{i,j=1}. \label{generate_R_by_Z}
\end{equation}
Let $\Theta$ be the Cholesky factorization of $R$ as $R = \Theta {\Theta}^{\top}$. Then we can generate the realizations of $Z(\omega,\vec{x})$ at the discrete points $\vec{x}$ without any truncation by
\begin{equation}
Z(\omega,\vec{x}) = \Theta Y, \label{realize_Z}
\end{equation}
where $Y:=(Y_1(w),...,Y_M(w))^{\top}$ is a $M \times 1$ vector of independent identically distributed standard Gaussian random variables. It is easy to verify that $\mathbb{E}[Z(\omega,x)] = \mathbb{E}[ \Theta Y ] = \Theta \mathbb{E}[ Y ] = \mathbf{0}$, and $\mathbb{E}[Z(\omega,\vec{x}),Z(\omega,\vec{x})^\top] = \mathbb{E}[ (\Theta Y)(\Theta Y)^{\top} ] = \Theta \mathbb{E}[ Y Y^{\top} ] \Theta^{\top} = \Theta \Theta^{\top} = R$. And the realizations of $K(\omega,x)$ are generated by the formulation \eqref{log_normal_K}. Some samples of the realizations of the random hydraulic conductivity $K$ will be displayed in the latter section.

\subsection{Monte Carlo methods}
The Monte Carlo method \cite{MC_Robert} is a classical method to calculate the numerical approximation of moments. In this paper, we only investigate the process to generate the expected value of $\phi_m$, $\vec{u}_s$ and $p_{s}$, which is easy to be used for the high order of moments.

For simplification, the symbol $Q$ is used to substitute the quantity of interest (QoI) of $\phi_m $, $\vec{u}_s$ and $p_{s}$. Let $Q_{\ell}(\omega,x)$ denote the finite element approximation of $Q(\omega,x)$ on the quasi-uniform triangulation mesh $\mathcal{T}_{\ell}$ with the mesh size $h_{\ell}$, and $Q^{i}_{\ell}(x)$ denote the realization of $Q_{\ell}(\omega,x)$ with the sample $\mathbb{K}(\omega^i,x)$. Then the approximation $\hat{Q}^{SL}_{\ell}(x)$ of the expected value of $Q$ by SLMC method with $N_{\ell}^{SL}$ samples $\{\mathbb{K}(\omega^i,x)\}^{N^{\ell}_L}_{i=1}$ is given as:
\begin{equation}
\hat{Q}^{SL}_{\ell}(x) = \frac{1}{N^{SL}_{\ell}}\sum_{i=1}^{N^{SL}_{\ell}}Q_{\ell}^i(x). \label{Monte_Carlo_FE_for_QoI}
\end{equation}
When no ambiguity arises, we may omit $x$ in $Q_{\ell}(x)$, $Q_{\ell}^i(x)$ and $\hat{Q}_{\ell}(x)$ for convenience.

The mean squared error of the SLMC method is:
\begin{equation}
  \begin{split}
        MSE(\hat{Q}^{SL}_L)&=\mathbb{E}[(\hat{Q}^{SL}_L-\mathbb{E}[Q])^2]\\
        &=\mathbb{E}[(\hat{Q}^{SL}_L-\mathbb{E}[Q_L]+\mathbb{E}[Q_L]-\mathbb{E}[Q])^2]\\
        &\leq 2\mathbb{E}[(\hat{Q}^{SL}_L-\mathbb{E}[Q_L])^2]+2\mathbb{E}[(\mathbb{E}[Q_L]-\mathbb{E}[Q])^2]\\
        &=\frac{2\mathbb{V}[Q_L]}{N^{SL}_L}+2(\mathbb{E}[Q_L]-\mathbb{E}[Q])^2.
  \end{split}
\label{Error_Monte_Carlo_FE_for_QoI}
\end{equation}

Then the error of SLMC method with a given norm $\|\cdot\|$ is bounded as
\begin{equation}
\| MSE(\hat{Q}^{SL}_L)\| \leq \frac{2\|\mathbb{V}[Q_L]\|}{N^{SL}_L}+2\|(\mathbb{E}[Q_L]-\mathbb{E}[Q])^2\|,
\label{normal_Error_Monte_Carlo_FE_for_QoI}
\end{equation}
i.e., the accuracy of SLMC method is based on the sampling error and the FEM error.

\subsection{Multi-level Monte Carlo methods}
The total computational cost $T^{SL}_c$ of single-level Monte Carlo is
\begin{equation}
T^{SL}_c = N^{SL}_{L}C_{L},
\label{TC_SLMC}
\end{equation}
where $C_{L}$ is the computational cost of one sample with mesh size $h_{L}$. $T_{c}^{SL}$ would be very high when $N^{SL}_{L}$ and $C_{L}$ are both very large. By the accuracy formulation \eqref{normal_Error_Monte_Carlo_FE_for_QoI} of SLMC method, the sampling error and the FEM error should be both small enough, if a small mean squared error is required. Thus $N^{SL}_{L}$ should be larger while the mesh size $h_L$ becomes smaller. On the other hand, $C_{L}$ increase exponentially as the mesh size $h_{L}$ becomes smaller. Thus the total computational cost increases very fast as mesh size $h_L$ become smaller. An efficient algorithm is needed to reduce the total computational cost. We adopt the multi-level Monte Carlo (MLMC) method.

By the linearity of the expectation operator
\begin{equation}
    \mathbb{E}[Q_{L}] = \mathbb{E}[Q_{0}] + \sum_{\ell = 1}^{L} \mathbb{E}[Q_{\ell}] - [Q_{\ell - 1}] =  \mathbb{E}[Q_{0}] + \sum_{\ell = 1}^{L} \mathbb{E}[Q_{\ell} - Q_{\ell - 1}].
    \label{MLMC_E}
\end{equation}

Then we can use the hierarchical meshes to construct the MLMC method to generate the expect value of $Q$. Let $\{\mathcal{T}_{\ell}\}_{\ell=0}^{L}$ be a sequence of quasi-uniform triangulation meshes with the mesh sizes $\{h_{\ell}\}_{\ell=0}^{L}$. These mesh sizes satisfy $h_{\ell}=h_0c_h^{-\ell}$, $\ell =0,1,2,\cdots,L$. And $\{N^{ML}_{\ell}\}_{\ell=0}^L$ are the numbers of samples with the mesh sizes $\{h_{\ell}\}_{\ell=0}^{L}$. Then the approximation $\hat{Q}_L^{ML}$ of the expected value by the MLMC method is given by:
\begin{equation}
     \hat{Q}_L^{ML} =\frac{1}{N^{ML}_0}\sum_{i=1}^{N^{ML}_0}Q_0^i
            +\sum_{\ell=1}^{L}\frac{1}{N^{ML}_\ell}\sum_{i=1}^{N^{ML}_\ell}(Q_{\ell}^i-Q_{\ell-1}^i),
    \label{MLMC_FE_for_QoI}
\end{equation}
and the corresponding mean squared error of the MLCM method with norm $\|\cdot\|$ is
\begin{equation}
    \begin{split}
        \|MSE(\hat{Q}_L^{ML})\|&=\|\mathbb{E}[(\hat{Q}_L^{ML}-\mathbb{E}[Q])^2]\| \\
        &\leq 2\frac{\| \mathbb{V}[Q_0] \| }{N^{ML}_0} + 2\sum_{\ell=1}^{L}\frac{ \| \mathbb{V}[Q_{\ell}-Q_{\ell-1}] \| }{N^{ML}_{\ell}}
         + 2\|(\mathbb{E}[Q_L]-\mathbb{E}[Q])^2\|.
    \end{split}
    \label{normal_MSE_MLMC_FE_for_QoI}
\end{equation}

For simplicity, let $Q_{-1}=0$, $h_{-1}=0$, $v_{\ell}=\| \mathbb{V}[Q_{\ell}-Q_{\ell-1}] \|$, $\ell=0,1,2,\cdots,L$, and $C_{\ell}$ be the computational cost of generating one sample of $Q_{\ell}-Q_{\ell-1}$, $\ell=0,1,\cdots,L$. Then the mean squared error is rewrote as
\begin{equation}
\|MSE(\hat{Q}_L^{ML})\| \leq
2\sum_{\ell=0}^L\frac{v_\ell}{N^{ML}_\ell} + 2\|(\mathbb{E}[Q_L]-\mathbb{E}[Q])^2\|.
\label{MSE_MLMC}
\end{equation}
And the total computational cost $T^{ML}_c$ is
\begin{equation}
T^{ML}_c=\sum_{\ell=0}^{L}N^{ML}_{\ell}C_{\ell} \ .
\label{TC_MLMC}
\end{equation}
By the mean squared error of SLMC method \eqref{normal_Error_Monte_Carlo_FE_for_QoI} and MLMC method \eqref{MSE_MLMC}, the accuracy of approximation of expected value is based on two parts, i.e., the sampling error and FEM error. The FEM error $\|(\mathbb{E}[Q_L]-\mathbb{E}[Q])^2\|$ is fixed when the mesh size $h_{L}$ is given. Thus the sampling error should be small enough with the given mesh size $h_{L}$. We substitute the sampling errors in SLMC method and MLMC method by:
\begin{equation}
e^{SL}_{L} =  \frac{\|\mathbb{V}[Q_L]\|}{N_L}, \text{\ \ and \ \ }
e^{ML}_L = \sum_{\ell=0}^L\frac{v_\ell}{N^{ML}_\ell} \ .
\label{MSE_MLMC_MC}
\end{equation}
For guaranteeing the accuracy of MLMC method is as same as SLMC method, the following relationship between two sampling errors should be ensured
\begin{equation}
e^{ML}_L \leq e^{SL}_{L}, \text{\ i.e., \ }
\frac{v_0}{N_0^{ML}}+\frac{v_1}{N_1^{ML}}+\cdots+\frac{v_L}{N_L^{ML}} \leq \frac{\|\mathbb{V}[Q_L]\|}{N^{SL}_L}. \label{guarante_ML_accuracy_SL}
\end{equation}

Then we show our strategy to generate the key parameters for MLMC method: the total number of levels $L$, and the number of samples at every level $\{N^{ML}_{\ell}\}_{\ell=0}^L$.

The total number of levels $L = \log_{c_h}^{h_0/h_L}$ depends on three variables: the mesh size decrease parameter $c_h$, the largest mesh size $h_0$ and the smallest mesh size $h_L$. The largest mesh size $h_0$ is constrained by the size of the physical area. The smallest mesh size $h_L$ depends on the accuracy of FEM as the practical problem required. Then $L$ is given after the setting $c_h = 2$.

The guideline in designing the number of samples at every level is minimizing computational cost under the given sampling error. Thus we introduce the optimization problem as follow:
\begin{equation}
    \begin{cases}
    & \text{Minimize}\  T^{ML}_c=N^{ML}_0C_0+N^{ML}_1C_1+\cdots+N^{ML}_LC_L,\\
    & \text{subject to}\  \frac{v_0}{N^{ML}_0}+\frac{v_1}{N^{ML}_1}+\cdots+\frac{v_L}{N^{ML}_L} = e^{ML}_L.
    \end{cases}
    \label{optimization_N_as_giveb_error}
\end{equation}
This optimization problem is solved by the method of Lagrangian multipliers:
\begin{equation}
    \begin{split}
    \mathcal{L} = &N^{ML}_0C_0+N^{ML}_1C_1+\cdots+N^{ML}_LC_L\\
                & + \lambda(\frac{v_0}{N^{ML}_0}+\frac{v_1}{N^{ML}_1}+\cdots+\frac{v_L}{N^{ML}_L} - e^{ML}_L).
    \end{split}
    \label{Lagrangian_for_optimization_N_as_given_error}
\end{equation}
Then the equations for $\{N_{\ell}\}_{\ell=0}^L$ are
\begin{equation}
    \begin{cases}
    & \frac{\partial\mathcal{L}}{\partial{N^{ML}_\ell}}=C_{\ell}-\lambda \frac{v_{\ell}}{(N^{ML}_{\ell})^{2}}=0,\ for\ \ell=0,1,\cdots,L,\\
    & \frac{\partial\mathcal{L}}{\partial{\lambda}}=\frac{v_0}{N^{ML}_0}+\frac{v_1}{N^{ML}_1}+\cdots+\frac{v_L}{N^{ML}_L} - e^{ML}_L=0.
    \end{cases}
    \label{equation_for_solve_optimization_problem}
\end{equation}
Then the number of samples at the every level is
\begin{equation}
N_{\ell}^{ML}=\sqrt{\frac{v_{\ell}}{C_{\ell}}}\left(\frac{\sqrt{v_0C_0}+\sqrt{v_1C_1}+\cdots+\sqrt{v_LC_L}}{e^{ML}_L}\right),
\label{number_of_samples_at_every_level}
\end{equation}
and the optimal computational cost is
\begin{equation}
T_c^{opt}=T_c^{ML}=\frac{\left(\sqrt{v_0C_0}+\sqrt{v_0C_0}+\cdots+\sqrt{v_LC_L}\right)^2}{e^{ML}_L}\ . \label{Total_C_MLMC}
\end{equation}
In the application of this strategy, the parameter $e^{ML}_L$ is given by the formula \eqref{guarante_ML_accuracy_SL}.

We assume $v_{\ell}=\mathcal{O}(h_{\ell}^{\beta})$ by the virtue of experience, and $C_{\ell}=\mathcal{O}(h_{\ell}^{-\gamma})$ because the number of information be calculated increase exponentially while the mesh size becomes smaller. Under the choice $c_h=2$, i.e., $h_\ell=h_0 2^{-\ell}$, $\ell=0,1,\cdots,L$, by the formula \eqref{number_of_samples_at_every_level}, for any $j > i$
\begin{equation}
\frac{N^{ML}_j}{N^{ML}_i}=\sqrt{\frac{C_i}{v_i}\cdot\frac{v_j}{C_j}}
=\sqrt{\mathcal{O}\left(\left(\frac{h_j}{h_i}\right)^{\beta+\gamma}\right)}
=\mathcal{O}\left(2^{-\frac{(j-i)(\beta+\gamma)}{2}} \right)
< 1.
\label{relation_of_number_of_samples_two_level_as_half_the_mesh}
\end{equation}
Thus the number of samples be calculated becomes smaller while mesh size becomes smaller. The decrease of $N_{\ell}$ is the reason why the MLMC method can reduce the total computational cost.

Since the computational cost of every sample with the mesh size $h_0$ is low, $v_{0}$ is easy to calculate by Monte Carlo method with low computational cost. Then $v_{\ell}$, $\ell=1,\cdots,L$, can be given by $v_{\ell}=\mathcal{O}(h_{\ell}^{\beta})$ with the corresponding parameter $\beta$. Thus how to determine parameter $\beta$ is a key problem for MLMC method. Our strategy is provided in the following section.

\subsection{Multi-grid methods}
The total computational cost depends on the number of samples and the computational cost of every sample. Since we have reduced the total computational cost in probability space by using the MLMC method to reduce the number of samples, it is a heuristic problem that can we also reduce the total computational cost in physical space by reducing computational cost of every sample. Inspired by the hierarchical meshes used in the MLMC method, we adopt the multi-grid (MG) method to reduce the computational cost in physical space.

In the physical space, the finite element method (FEM) is chosen to construct the discrete form of weak formulation \eqref{weak_formulation} under the given samples of hydraulic conductivity. We adopt the Taylor-Hood element in the conduit domain, and the quadratic element in the porous media domain. Then for every given sample of hydraulic conductivity $\mathbb{K}(\omega,x)$, the weak formulation \eqref{weak_formulation} is discretized into the following matrix-vector form
\begin{equation}
\mathbf{L}\mathbf{x} = \mathbf{b},
\ \ \mathbf{L}=\begin{pmatrix} A_m & B_1 & 0\\ B_2 & A_s & B'_p \\ 0 & B_p & 0 \end{pmatrix},
\ \ \mathbf{x} = \begin{pmatrix} \phi_m \\ \vec{u}_s \\ p \end{pmatrix},
\ \ \mathbf{b} = \begin{pmatrix} b_m \\ \vec{b}_s \\ 0 \end{pmatrix}.
\label{algebraic_equations}
\end{equation}
where $A_{m}$ is the discretization of $ g\int_{D_m}( \mathbb{K}\nabla\phi_m )\cdot \nabla \psi_mdx$, $B_{1}$ is the discretization of $ - g\int_{\Gamma_I} (\vec{u}_s\cdot\vec{n}_s) \psi_m d\Gamma_I $, $B_{2}$ is the discretization of $ \int_{\Gamma_I} g\phi_m\vec{v}_s\cdot\vec{n}_s + \frac{\alpha\nu\sqrt{d}}{\sqrt{\text{trace}(\Pi)}}P_{\tau}(\mathbb{K}\nabla\phi_m)\cdot\vec{v}_sd\Gamma_I $, $A_{s}$ is the discretization of $ \int_{D_s}2\nu\mathbb{D}(\vec{u}_s):\mathbb{D}(\vec{v}_s)dx + \int_{\Gamma_I} \frac{\alpha\nu\sqrt{d}}{\sqrt{\text{trace}(\Pi)}}P_{\tau}(\vec{u}_s)\cdot\vec{v}_sd\Gamma_I$, $B'_{p}$ is the discretization of $ \int_{D_s}p_s\nabla\cdot\vec{v}_s dx $, $b_{m}$ is the discretization of $ g\int_{D_m}f_m\psi_mdx $, and $\vec{b}_{s}$ is the discretization of $ \int_{D_s}\vec{f}_s\cdot\vec{v}_sdx
+\int_{\Gamma_I} gz\vec{v}_s\cdot\vec{n}_sd\Gamma_I $.

Since there exists a zeros block in the diagonal of stiffness matrix $\mathbf{L}$, we can not directly solve the algebraic equations $\eqref{algebraic_equations}$ by iterative method such as Gauss-Seidel method. Inspired by the multi-grid method for Stokes equations, we adopt the efficient least squares commutator distributive Gauss-Seidel (LSC-DGS) relaxation \cite{MWang_LChen_MG_2013, LChen_MG_2015} in this paper. The right-side operator $\mathbf{M}$ is given as:
\begin{equation}
\mathbf{M} = \begin{pmatrix} I & 0 & 0\\ 0 & I & B'_p \\ 0 & 0 & -(B_p B'_p)^{-1}B_pA_s B'_p \end{pmatrix}.
\label{M_LSCDGS}
\end{equation}
Multiplying $\mathbf{L}$ with $\mathbf{M}$ yields
\begin{equation*}
\mathbf{L}\mathbf{M} = \begin{pmatrix} A_m & B_1 & B_1B'_p\\ B_2 & A_s & W \\ 0 & B_p & B_p B'_p \end{pmatrix}, \ \text{with} \ W = \left(I-B'_p(B_p B'_p)^{-1}B_p\right)A_sB'_p.
\end{equation*}
By $\mathbf{S} := \mathbf{L}\mathbf{M} $ and $\mathbf{y} := \mathbf{M^{-1}x} $, the equivalent algebraic equations are given as
\begin{equation}
\mathbf{Sy} = \mathbf{b}.
\label{algebraic_equations_y}
\end{equation}

The standard Gauss-Seidel method is proposed to solve the equivalent algebraic equations \eqref{algebraic_equations_y}. And the following $\mathcal{V}$-cycle multi-grid method is applied to reduce the computational cost in physical space. As same as in the MLMC method, the hierarchical quasi-uniform triangulation meshes are $\mathcal{T}_{\ell}$ with the mesh sizes $h_{\ell}=h_0c_h^{-\ell}$, $\ell =0,1,2,\cdots,L$. Then the $\mathcal{V}$-cycle multi-grid method on the mesh $\mathcal{T}_{\ell}$ with the mesh size $h_{\ell}$ is given as:
\begin{algorithm}
$\mathbf{y} \leftarrow$ $\mathcal{V}$-cycle$(\mathbf{S}, \mathbf{b}, \ell)$
    \begin{enumerate}[(1)]
      \item Relax $\lambda_1$ times on the fine mesh $h = h_{\ell}$ with the initial gauss $\mathbf{y}$ to reach $\mathbf{y}^{h}$.
      \item Obtain the residual on the fine mesh as
          \begin{equation*}
          r^h = \mathbf{b} - \mathbf{Sy}^{h},
          \end{equation*}
          and restrict the residual from the fine mesh $h$ to the coarse mesh $H = h_{\ell-1}$ by $r^{H} = \mathcal{R}^{H}_{h} r^h$, where $\mathcal{R}^{H}_{h}$ is the restriction matrix.
      \item Solve the corrected error from the residual equation on the coarse mesh $H$:
            \begin{itemize}
              \item If $\ell = 1$, use a direct or fast iterative method to solve $ \mathbf{S}^H e^H = r^H $;
              \item If $\ell > 1$, use the $\ell$-grid method to solve $\mathbf{S}^H e^H = r^H $ from a zero initial gauss on the mesh $\mathcal{T}_{\ell-1}$ by $e^H \leftarrow$ $\mathcal{V}$-cycle$(\mathbf{S}^{H}, r^{H}, \ell-1)$;
            \end{itemize}
          where $\mathbf{S}^H$ is the approximation of $\mathbf{S}$ on the coarse mesh.
      \item Prolongate the corrected error form coarse mesh $H$ to the fine mesh $h$ by $e^h = \mathcal{I}^h_H e^H$, where $\mathcal{I}^h_H$ is the interpolation matrix. And correct the approximation by
          \begin{eqnarray*}
          \mathbf{y}^{new} = \mathbf{y}^{h} + e^h.
          \end{eqnarray*}
      \item Relax $\lambda_2$ times on the fine mesh $h$ with the initial gauss $\mathbf{y}^{new}$.
    \end{enumerate}
\label{algorithm_LSCDGS}
\end{algorithm}

We can replace the $\mathcal{V}$-cycle by $\mathcal{W}$-cycle or $\mathcal{F}$-cycle. Furthermore, the solutions be calculated on the coarse mesh in MLMC method could be used as the initial gauss on the fine mesh in MG method to further reduce the computational cost. Then the following multi-gird multi-level Monte Carlo (MGMLMC) method is developed to reduce the computational cost both in the probability space and in physical space.

\begin{algorithm} multi-grid multi-level Monte Carlo method
    \begin{enumerate}[(1)]
      \item On the mesh grid $\mathcal{T}_{0}$ with the mesh size $h_{0}$, for the 1st to the $N^{ML}_{0}$ sample of hydraulic conductivity $\mathbb{K}(\omega,x)$, solve the numerical approximations $ Q_{0}^{i} $ by standard Gauss-Seidel with the initial gauss $\mathbf{0}$, $i=1, 2, \cdots, N_{0}^{ML}$;
      \item On the mesh grid $\mathcal{T}_{1}$ with the mesh size $h_{1}$, for the 1st to the $N^{ML}_{1}$ sample of hydraulic conductivity $\mathbb{K}(\omega,x)$, solve the numerical approximations $ Q_{1}^{i} $ by $\mathcal{V}$-cycle $2$-grid method with the initial gauss $ \mathcal{I}^{h}_{H} Q_{0}^{i} $, $i=1, 2, \cdots, N_{1}^{ML}$, where $h = h_{1}$ and $H = h_{0}$;
      \item  \hspace{4 cm}  $\cdots\cdots$
      \item On the mesh grid $\mathcal{T}_{\ell}$ with the mesh size $h_{\ell}$, for the 1st to the $N^{ML}_{\ell}$ sample of hydraulic conductivity $\mathbb{K}(\omega,x)$, solve the numerical approximations $ Q_{\ell}^{i} $ by $\mathcal{V}$-cycle $(\ell + 1)$-grid method with the initial gauss $ \mathcal{I}^{h}_{H} Q_{\ell-1}^{i} $, $i=1, 2, \cdots, N_{\ell}^{ML}$, where $h = h_{\ell}$ and $H = h_{\ell - 1}$;
      \item \hspace{4 cm} $\cdots\cdots$
      \item On the mesh grid $\mathcal{T}_{L}$ with the mesh size $h_{L}$, for the 1st to the $N^{ML}_{L}$ sample of hydraulic conductivity $\mathbb{K}(\omega,x)$, solve the numerical approximations $ Q_{L}^{i} $ by $\mathcal{V}$-cycle $(L+1)$-grid method with the initial gauss $ \mathcal{I}^{h}_{H} Q_{L-1}^{i} $, $i=1, 2, \cdots, N_{L}^{ML}$, where $h = h_{L}$ and $H = h_{L-1}$.
    \end{enumerate}
\label{algorithm_MGMLMC}
\end{algorithm}

\subsection{Computational cost of SLMC method and MGML method}
The numerical error is estimated as follow:
\begin{propo}\label{error_estimate}
Under the assumption \eqref{K_assumption_1} or \eqref{K_assumption_2}, the solutions of problem \eqref{weak_formulation} satisfy the following discrete error estimate
\begin{equation}
\| \underline{u} - \underline{u}_{\ell} \|_{X^{r}} + \| p_{s} - p_{\ell, s} \|_{\mathcal{L}^2(D_s)} \leq C h_{\ell}^{1-r} \Big( \| \underline{u} \|_{X^{1}} + \| p_{s} \|_{\mathcal{L}^2(D_s)} \Big),
\label{FEM_error_Stochastic}
\end{equation}
where $h_{\ell}$ is the mesh size of the given quasi-uniform triangulation mesh $\mathcal{T}_{\ell}$, $r = 0, 1$, and $\|\cdot\|_{X^{r}} $ is the norm of $\underline{u}$ defined in \eqref{norm_u_phi}.
\end{propo}
\begin{proof}
Based on the analysis in \cite{YCao_MGunzburger_XHu_FHua_XWang_WZhao_1, YCao_MGunzburger_FHua_XWang_1, WJLayton_FSchieweck_IYotov_1}, we have
\begin{align*}
&\| \vec{u}_{s}(\omega, \cdot) - \vec{u}_{\ell,s}(\omega, \cdot) \|_{\mathbf{H}^{r}(D_s)} + \| \phi_{m}(\omega, \cdot) - \phi_{\ell, m}(\omega, \cdot)  \|_{H^{r}(D_m)} + \| p_{s}(\omega, \cdot) - p_{\ell,s}(\omega, \cdot) \|_{L^2(D_s)} \\
& \leq C h_{\ell}^{1-r} \Big( \| \vec{u}_{s}(\omega, \cdot) \|_{\mathbf{H}^{1}(D_s)} + \| \phi_{m}(\omega, \cdot) \|_{H^{1}(D_m)} + \| p_{s}(\omega, \cdot) \|_{L^2(D_s)} \Big), \hspace{.1cm} \text{a.e.\ }  \omega\in\Omega,
\end{align*}
with $r = 0, 1$. Then the assertion follows with the above conclusion by the the definition of the norm $\|\cdot\|_{X^{r}}$ in \eqref{norm_u_phi}.
\end{proof}

Then the numerical errors of SLMC method and MGML method are bounded by the mesh size $h_{\ell}$ and the number of samples $N_{\ell}$.

\begin{lemma}\label{error_h_N_lemma}
Under the assumption \eqref{K_assumption_1} or \eqref{K_assumption_2}, the error bounds of SLMC method \eqref{Monte_Carlo_FE_for_QoI} and MLMC \eqref{MLMC_FE_for_QoI} for the problem \eqref{weak_formulation} are given as follows
\begin{eqnarray}
\| \mathbb{E}[\underline{u}] - \hat{\underline{u}}^{SL}_{\ell} \|_{X^{0}} + \| \mathbb{E}[p_{s}] - \hat{p}^{SL}_{s, \ell} \|_{\mathcal{L}^2(D_s)}
     &\leq& C(\underline{u}, p_{s}) \Big( h_{\ell} + (N_{\ell}^{SL})^{-1/2} \Big), \label{error_h_N_SL}\\
\| \mathbb{E}[\underline{u}] - \hat{\underline{u}}^{ML}_{L} \|_{X^{0}} + \| \mathbb{E}[p_{s}] - \hat{p}^{ML}_{s, L} \|_{\mathcal{L}^2(D_s)}
    &\leq& C(\underline{u}, p_{s}) \Big( h_{L} + \sum_{\ell=0}^{L} h_{\ell} ( N_{\ell}^{ML} )^{-1/2} \Big), \label{error_h_N_ML}
\end{eqnarray}
where $C$ depends on $\underline{u}$ and $p_{s}$, $h_{\ell}$ is the mesh size of the quasi-uniform triangulation mesh $\mathcal{T}_{\ell}$, and $\hat{\underline{u}}^{SL}_{\ell}$, $\hat{p}^{SL}_{s, \ell}$, $\hat{\underline{u}}^{ML}_{\ell}$, $\hat{p}^{ML}_{s, \ell}$ are the approximations of expect value by SLMC method, MLMC method, $\ell = 0, 1, \cdots, L$.
\end{lemma}
\begin{proof}
For simplification, let the symbol $Q$ substitute the variables $\phi_{m}$, $u_{s}$ or $p_{s}$. And let $\mathcal{L}(V)$ denote the corresponding space of $Q$. Thus $\mathcal{L}(V)$ may denote $\mathcal{H}^{1}(D_m)$, $\mathcal{L}^{2}(D_s)$ or $\vec{\mathcal{H}}^{1}(D_{s})$, i.e., $V$ may be $H^{1}(D_{m})$, $L^{2}(D_{s})$ or $\mathbb{H}^{1}(D_{s})$, which depends on the choice of $Q$. Then we can analyse the error of the approximation of expect value of $\phi_{m}$, $u_{s}$ or $p_{s}$ by analysing $ \mathbb{E}[Q] -  \hat{Q}^{SL}_{\ell} $ with the norm $\| \cdot \|_{\mathcal{L}(V)}$ as follow.
\begin{equation}
  \begin{split}
        \| \mathbb{E}[Q] -  \hat{Q}^{SL}_{\ell} \|_{\mathcal{L}(V)}
        & = \| \mathbb{E}[Q] - \mathbb{E}[Q_{\ell}] + \mathbb{E}[Q_{\ell}] - \hat{Q}^{SL}_{\ell} \|_{\mathcal{L}(V)} \\
        & \leq \| \mathbb{E}[Q] - \mathbb{E}[Q_{\ell}] \|_{\mathcal{L}(V)} + \| \mathbb{E}[Q_{\ell}] - \hat{Q}^{SL}_{\ell} \|_{\mathcal{L}(V)}.
  \end{split}
\end{equation}

For $ \| \mathbb{E}[Q] - \mathbb{E}[Q_{\ell}] \|_{\mathcal{L}(V)} $, we have
\begin{equation}
  \begin{split}
        \| \mathbb{E}[Q] - \mathbb{E}[Q_{\ell}] \|_{\mathcal{L}(V)}^{2}
        & = \| \mathbb{E}[Q - Q_{\ell}] \|_{\mathcal{L}(V)}^{2}
         = \mathbb{E}\big[ \| \mathbb{E}[Q - Q_{\ell}] \|_{V}^{2} \big] \\
        & = \| \mathbb{E}[Q - Q_{\ell}] \|_{V}^{2}
        \leq \mathbb{E} [ \| Q - Q_{\ell} \|_{V}^{2} ] \\
        &  =  \| Q - Q_{\ell} \|_{\mathcal{L}(V)}^{2}.
  \end{split}
\end{equation}

For $ \| \mathbb{E}[Q_{\ell}] - \hat{Q}^{SL}_{\ell} \|_{\mathcal{L}(V)} $, we have
\begin{equation}
  \begin{split}
        \| \mathbb{E}[Q_{\ell}] - \hat{Q}^{SL}_{\ell} \|_{\mathcal{L}(V)}^{2}
        & = \mathbb{E}\bigg[ \| \mathbb{E}[Q_{\ell}] - \frac{1}{N_{\ell}^{SL}} \sum_{i=1}^{N_{\ell}^{SL}} Q_{\ell}^{i}  \|_{V}^{2} \bigg] \\
        &= \frac{1}{(N_{\ell}^{SL})^{2}} \mathbb{E}\bigg[ \| \sum_{i=1}^{N_{\ell}^{SL}}\big( \mathbb{E}[Q_{\ell}] -  Q_{\ell}^{i} ) \|_{V}^{2} \big] \\
        & \leq \frac{1}{(N_{\ell}^{SL})^{2}} \mathbb{E}\bigg[ \sum_{i=1}^{N_{\ell}^{SL}} \| \mathbb{E}[Q_{\ell}] -  Q_{\ell}^{i} \|_{V}^{2} \bigg] \\
        & = \frac{1}{N_{\ell}^{SL}} \mathbb{E}\big[ \| \mathbb{E}[Q_{\ell}] -  Q_{\ell} \|_{V}^{2} \big] \\
        & \leq \frac{1}{N_{\ell}^{SL}} \| Q_{\ell} \|_{\mathcal{L}(V)}^{2}.
  \end{split}
\end{equation}
The last inequality is based on $\mathbb{E}[ (\mathbb{E}[Q_{\ell}] -  Q_{\ell})^2 ] = \mathbb{E}[ (Q_{\ell})^2 ] - (\mathbb{E}[ Q_{\ell} ])^2 \leq \mathbb{E}[ (Q_{\ell})^2 ]$.

Thus we obtain
\begin{equation}
  \begin{split}
        \| \mathbb{E}[Q] -  \hat{Q}^{SL}_{\ell} \|_{\mathcal{L}(V)}
        \leq (N^{SL}_{\ell})^{-1/2} \| Q_{\ell} \|_{\mathcal{L}(V)}
        + \| Q - Q_{\ell} \|_{\mathcal{L}(V)}.
  \end{split}
\end{equation}

Then by the Proposition \ref{error_estimate}, we have
\begin{equation*}
  \begin{split}
        &\| \mathbb{E}[\underline{u}] - \hat{\underline{u}}^{SL}_{\ell} \|_{X^{0}} + \| \mathbb{E}[ p_{s} ] - \hat{p}^{SL}_{s, \ell} \|_{\mathcal{L}^2(D_s)} \\
        & \leq (N^{SL}_{\ell})^{-1/2}\| \underline{u}_{\ell} \|_{X^{0}} + \| \underline{u} - \underline{u}_{\ell} \|_{X^{0}} + (N^{SL}_{\ell})^{-1/2} \| p_{s, \ell} \|_{\mathcal{L}^2(D_s)} + \| p_{s} - p_{s, \ell} \|_{\mathcal{L}^2(D_s)} \\
        & \leq C \Big(h_{\ell} + (N^{SL}_{\ell})^{-1/2} \Big)\Big( \| \underline{u} \|_{X^{0}} + \| \underline{u} \|_{X^{1}} + \| p_{s} \|_{\mathcal{L}^2(D_s)} \Big) \\
        & = C(\underline{u}, p_{s}) \Big( h_{\ell} + (N_{\ell}^{SL})^{-1/2} \Big).
  \end{split}
\end{equation*}
where $C(\underline{u}, p_{s})$ depends on $\| \underline{u} \|_{X^{0}}$, $\| \underline{u} \|_{X^{1}}$ and $\| p_{s} \|_{\mathcal{L}^2(D_s)}$.

Because the idea to prove the assertion of MLMC method is as same as that in the proof of the assertion of SLMC method, we skip it.
\end{proof}

By equilibrating the sampling error in probability space and the FEM error in physical space, we have the following two conclusions based on the conclusions in Lemma \ref{error_h_N_lemma}.
\begin{eqnarray}
( e^{SL}_{\ell} )^{1/2} &=& \mathcal{O}\big( (N_{\ell}^{SL})^{-1/2} \big) = \mathcal{O}\big( h_{\ell} \big), \label{error_bound_SL}\\
( e^{ML}_L )^{1/2} &=& \mathcal{O}\Big( \sum_{\ell=0}^{L} h_{\ell} ( N_{\ell}^{ML} )^{-1/2}  \Big) = \mathcal{O}\big( h_{L} \big). \label{error_bound_ML}
\end{eqnarray}
The formula \eqref{error_bound_SL} is the relationship between the numbers of samples $ N_{L} $ and the mesh sizes $ h_{L}$ in the SLMC method, which is based on the conclusion \eqref{error_h_N_SL}. And the formula \eqref{error_bound_SL} is the relationship between the numbers of samples $\{ N_{\ell} \}_{\ell = 0}^{L}$ and the mesh sizes $\{ h_{\ell} \}_{\ell = 0}^{L}$ in the MLMC method, which is based on the conclusion \eqref{error_h_N_ML}.

In the SLMC method, by \eqref{error_bound_SL}, it is easy to see that the number of samples $N_{L}$ on the finest mesh is determined by the mesh size $h_{L}$, then the computational cost is distinct. In the MLMC method, the number of samples $N_{\ell}$ on every level is determined by the formula \eqref{number_of_samples_at_every_level}, since the sampling error is bounded by the formula \eqref{error_bound_ML}, then the computational cost is also distinct.

\begin{theorem}
Under the assumption \eqref{K_assumption_1} or \eqref{K_assumption_2}, for the problem \eqref{weak_formulation}, if we choose the SLMC method \eqref{Monte_Carlo_FE_for_QoI} on the triangulation mesh $\mathcal{T}_{L}$ with the mesh size $h_{L}$, or the MGMLMC method Algorithm \ref{algorithm_MGMLMC} on the hierarchical quasi-uniform triangulation meshes $\{\mathcal{T}_{\ell}\}_{\ell=0}^{L}$ with mesh sizes $h_\ell=h_0 2^{-\ell}, \ell=0,1,\cdots,L$ to solve the approximations of expect value, we can evaluate the computational cost as follows:
\begin{eqnarray}
T^{SL}_{c} &=& \mathcal{O}\Big( M_{L}^{ 2+\frac{2}{d} } \Big), \\
T^{MGML}_{c} &=& \mathcal{O}\Big( M_{L}^{ 1+\frac{2-\beta}{d} } 2^{\frac{(L+1)\beta}{2}} \log^{M_{L}}  \Big),
\end{eqnarray}
where $M_{L}$ is the number of information be calculated for one sample on the mesh $\mathcal{T}_{L}$ with the mesh size $h_{L}$, $d$ is the dimension of the physical space, and $\beta$ is the decrease rate of the variance. Furthermore, the ratio of the computational cost of SLMC method and MGMLMC method is given as
\begin{equation}
T^{SL}_{c}/T^{MGML}_{c} = \mathcal{O}\Big( 2^{Ld + \frac{\beta(L-1)}{2} } /(Ld) \Big).
\end{equation}
\end{theorem}

\begin{proof}
Under the assumption $M_{\ell} = \mathcal{O}\big( h_{\ell}^{-d} \big)$, $M_{\ell} = \mathcal{O}\big( M_{L}2^{(\ell-L)d} \big)$ is given by the setting $h_\ell=h_0 2^{-\ell}$. Since the standard Gauss-Seidel method is chosen to solve the algebraic equations \eqref{algebraic_equations_y} in SLMC method, the computational cost $C_{\ell}$ with the mesh size $h_{\ell}$ is $C_{\ell} = \mathcal{O}\big( M^{2}_{L}2^{2(\ell-L)d} \big)$.

For the SLMC method, the number of samples on mesh $\mathcal{T}_{L}$ with mesh size $h_{L}$ is $N_{L}^{SL} = \mathcal{O}\big( h_{L}^{-2} \big) = \mathcal{O}\big( M_{L}^{2/d} \big)$, by the bound of sampling error in \eqref{error_bound_SL}. Then the computational cost of SLMC method is $T^{SL}_{c} = N_{L}^{SL}C_{L} = \mathcal{O}\Big( M_{L}^{ 2+\frac{2}{d} } \Big)$.

For the MLMC method, the bound of sampling error is $e_{L}^{ML} = \mathcal{O}\big( h_{L}^{2} \big) = \mathcal{O}\big( M_{L}^{-2/d} \big)$ by \eqref{error_bound_ML}. Then by \eqref{number_of_samples_at_every_level} the number of samples at the initial level is
\begin{equation}
N_0^{ML} = \sqrt{\frac{v_0}{C_0}}\left(\frac{\sqrt{v_0C_0}+\sqrt{v_1C_1}+\cdots+\sqrt{v_LC_L}}{e^{ML}_L}\right)
= \mathcal{O}\Big( M_{L}^{ \frac{2-\beta}{d} } 2^{ (d + \frac{\beta}{2})L } \Big).
\end{equation}
And by the relationship between the numbers of samples at two different levels \eqref{relation_of_number_of_samples_two_level_as_half_the_mesh}, the number of samples on mesh $\mathcal{T}_{\ell}$ with mesh size $h_{\ell}$ is
\begin{equation}
N^{ML}_{\ell}=N^{ML}_{0}\sqrt{\frac{C_0}{v_0}\cdot\frac{v_{\ell}}{C_{\ell}}} = \mathcal{O}\Big( M_{L}^{ \frac{2-\beta}{d} } 2^{ (d + \frac{\beta}{2})(L - \ell) } \Big).
\end{equation}
In this paper, we adopt the $\mathcal{V}$-cycle multi-grid methods. Then the computational cost \cite{Briggs_MG_2000, 2007_Strang} on mesh $\mathcal{T}_{\ell}$ with mesh size $h_{\ell}$ is
\begin{equation}
C_{\ell}^{MG} = \mathcal{O}\big( M_{L}2^{(\ell-L)d}( \log^{M_{L}} + (\ell - L)d\log^{2}  ) \big) \leq \mathcal{O}\big( M_{L}2^{(\ell-L)d}\log^{M_{L}} \big).
\end{equation}
Then the computational cost of MGMLMC method is
\begin{equation}
T^{MGML}_{c} = \sum_{\ell = 0}^{L} N_{\ell}^{ML} C_{\ell}^{MG} = \mathcal{O}\big( M_{L}^{1+\frac{2-\beta}{d}} 2^{(L+1)\frac{\beta}{2}}\log^{M_{L}} \big).
\end{equation}

By $h_{L} = h_0 2^{-L}$ and $M_{L} = \mathcal{O}\big( h_{L}^{-d} \big)$, we have
\begin{equation}
\frac{ T^{SL}_{c} }{ T^{MGML}_{c} } = \frac{ \mathcal{O}\big( 2^{Ld(2 + \frac{2}{d} )} \big) } { \mathcal{O}\big( 2^{ Ld(1 + \frac{2-\beta}{d}) + \frac{(L+1)\beta}{2} } Ld \big)  } = \mathcal{O}\Big( 2^{Ld + \frac{\beta(L-1)}{2} } /(Ld) \Big).
\end{equation}
\end{proof}

\section{Numerical experiments}
In this section, we use numerical experiments to demonstrate both the features of the MGMLMC method and the theoretical conclusion. The first part is to generate the realizations of random hydraulic conductivity $\mathbb{K}$ by the grid based method. The second part is to determine the parameters $\beta$, which will be used to calculate the $\{N^{ML}_{\ell}\}_{\ell=0}^L$ in MLMC method. The last part is to provide the numerical results in detail.

We assume that the domain $D_{ms}$ consists of two rectangles, the upper rectangle is the porous media domain $D_m = (0,1)\times (0,0.75)$, and the other rectangle is the conduit domain $D_s = (0,1)\times(-0.25,0)$, shown as the Figure \ref{two_rectangles}. The whole domain $D_{ms}=D_m\cup D_s$ with the interface $\Gamma_I = (0,1)\times\{0\}$. The boundary are $\Gamma_{m} = \{0,1\}\times(0,0.75) \cup (0, 1)\times\{0.75\} $ and $\Gamma_{s} = \Gamma_{s_{1}} \cup \Gamma_{s_{2}} \cup \Gamma_{s_{3}} $, where $\Gamma_{s_{1}}=\{0\}\times(-0.25, 0), \Gamma_{s_{2}}=(0, 1)\times \{-0.25\}, \Gamma_{s_{3}}=\{1\}\times(-0.25, 0)$. For simplicity, let $g=1$, $z=0$, $\alpha=1$, $\nu=1$ and $\mathbb{K}(\omega,x)=e^{Z(\omega,x)}\mathbb{I}$. The covariance function of $Z$ is $r (x, y )= r\big( (x_1,x_2),(y_1,y_2) \big)=0.1e^{-\frac{|x_1-y_1|}{0.2}-\frac{|x_2-y_2|}{0.2}}$.

\begin{figure}[H]
\centering
\includegraphics[height=2in,width=3in]{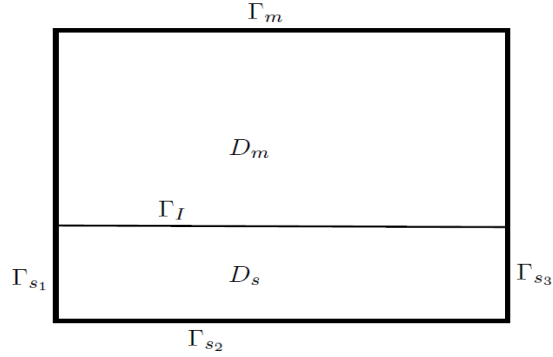}
\caption{A sketch of two rectangles domain.}
\label{two_rectangles}
\end{figure}

\subsection{The realizations of random hydraulic conductivity}
Because the diagonal matrix is given as $\mathbb{K}(\omega,x)=e^{Z(\omega,x)}\mathbb{I}$, it is natural to generate the realizations of $K(\omega,x)=e^{Z(\omega,x)}$, and then copy the realizations $d$ times to construct $\mathbb{K}(\omega,x)$. As the hierarchical meshes are used in MLMC method, the realizations of $K(\omega,x)=e^{Z(\omega,x)}$ could be first generated on the finest mesh by the grid based method. And then the realizations on the coarse mesh can be chosen as a subset of the realizations on finest mesh. Then the consistency of hydraulic conductivity $K(\omega,x)$ on every mesh could be ensured.

Because the Gauss quadrature points are the key points in the performance of finite element method, we calculate the value of the approximation of the hydraulic conductivity $K(\omega,x)$ on the Gauss quadrature points by the grid based method. In every triangle of the triangulation mesh $\mathcal{T}_{L}$, Gauss quadrature rule is applied with seven points and degree of precision three.

Four realizations of $K(\omega, x)$ are illustrated in Figure \ref{K_samples}, by which the randomness of hydraulic conductivity is exhibited.
\begin{figure}[H]
\centering
\includegraphics[width=6in]{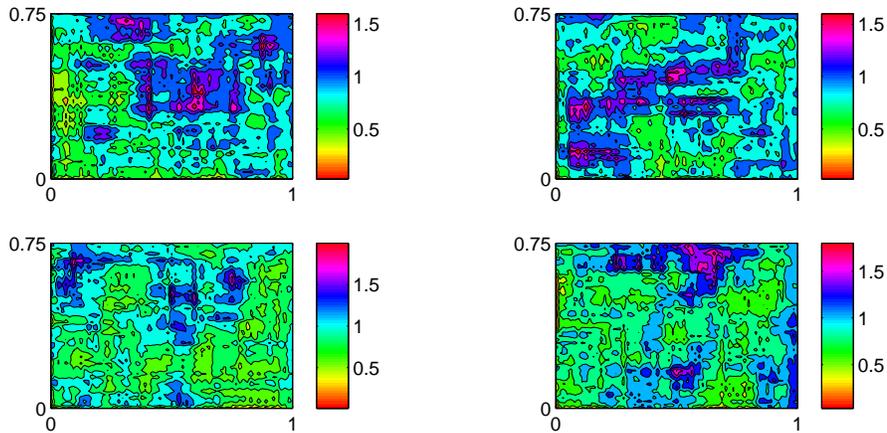}
\caption{4 samples of random hydraulic conductivity $K(\omega, x)$.}
\label{K_samples}
\end{figure}

\subsection{Determination of the parameter $\beta$ in MLMC method}
As we have discussed in the analysis for the number of samples at every level, how to determine the parameter $\beta$ is a key problem in the numerical implementation of MLMC method. The $\beta$ can be approximated by using the variances $\{v_{\ell}\}_{\ell=0}^L$, which are calculated by SLMC method with mesh size $\{ h_{\ell} \}_{\ell=0}^{L}$. Thus the computational cost of calculating the variances $\beta$ is more expensive than that of SLMC method with mesh size $h_{L}$, which contradicts the purpose of MLMC method. Then a practical method is needed to determine the parameter $\beta$ without calculating the variances $\{v_{\ell}\}_{\ell=0}^L$ by SLMC method at every level. In this paper, we develop the following method to calculate the parameter $\beta$.

Since the random hydraulic conductivity is only a parameter in the Darcy domain and on the interface, the Stokes equations in the coupled problem can be regarded as a boundary condition for a stochastic Darcy problem. We assume that the $\beta$ in the stochastic Darcy problem is an approximation for the $\beta$ in the stochastic Stokes-Darcy problem with the same random hydraulic conductivity. Compared with the coupled stochastic Stokes-Darcy problem, the domain of the stochastic Darcy problem is smaller and the computational cost of every sample is cheaper. Furthermore, we can also use the multi-grid method to reduce the computational cost in generating the approximation of $\beta$.

The stochastic Darcy problem is given by
\begin{equation}
    \begin{cases}
     & -\nabla \cdot \big( \mathbb{K}(\omega,x)\nabla \phi(\omega,x) \big) = f(\omega, x), \hspace{.2cm} (\omega,x)\in \Omega\times D_m, \\
     & \phi(\omega,x) = 0, \hspace{.2cm} (\omega,x)\in \Omega\times \partial D_m.
    \end{cases}
    \label{stochastic_Darcy}
\end{equation}
where $f(\omega, x)$ is a piecewise constant approximation of white noise, i.e.,
\begin{equation}
f(x) = \sigma \sum_{i=1}^{I}\frac{1}{\sqrt{V_{i}}}\chi_{i}(x)X_{i}(\omega), \hspace{.2cm} x\in D_m. \label{white_noise_f}
\end{equation}
Here $\sigma$ is a given constant, $V_{i}$ is the volume of non-overlapping tessellation $\{ D_{i} \}_{i=1}^{I}$ as $D_{m} = \cup_{i=1}^{I} D_{i}$, $\chi_{i}(x)$ is the indicator function corresponding to $D_{i}$, and $\{ X_{i}(\omega) \}_{i=1}^{I}$ is a given set of independent identically distributed standard Gaussian random variables.

Given the $\sigma$ and $\{ X_{i}(\omega) \}_{i=1}^{I}$ in \eqref{white_noise_f}, it is easy to calculate $\beta$ with a given norm $\|\cdot\|$. For each $\sigma=0.02, 0.8, 1.2$, we choose 40 samples of $f$. For every given $\sigma$ and one sample of $f$, three $\beta$ are calculated with $\|\cdot\|_{L_2}$, $\|\cdot\|_{L_{\infty}}$ and $\|\cdot\|_{H_1}$ norms. The results of $\beta$ with each sample of $f$ and the choice of $\sigma$ are exhibited in Figure \ref{beta_f}. And the mean values of $\beta$ with 40 samples of $f$ are shown in Table \ref{beta_sigma}. One can see that the mean value of the $\beta$ changes only in a small range when $\sigma$ becomes larger. Thus the parameter $\beta$ is given as 2.02, 1.65, 1.30 w.r.t. $\|\cdot\|_{L_2}$, $\|\cdot\|_{L_{\infty}}$, $\|\cdot\|_{H_1}$ norms.

\begin{figure}[H]
\centering
\includegraphics[width=6in]{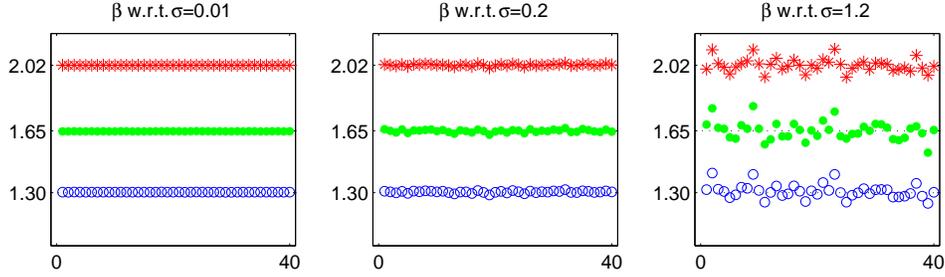}
\caption{0, 40 on x-label in every subgraph is the index of 40 samples of $f$. \\The red star \textcolor{red}{$\ast$} are beta with $\|\cdot\|_{L_2}$ norm, the green dot \textcolor{green}{$\cdot$} are beta with \\ $\|\cdot\|_{L_{\infty}}$ norm, and the blue circle \textcolor{blue}{$\circ$} are beta with $\|\cdot\|_{H_{1}}$ norm, }
\label{beta_f}
\end{figure}

\begin{table}[htbp]
    \centering
    \caption{Mean values of $\beta$ with different $\sigma$ and norm}
    \label{beta_sigma}
    \begin{tabular}{c|c|c|c}
    \hline
     \    $\sigma$ & 0.02 & 0.8 & 1.2  \\
     \hline
     $\|\cdot\|_{L_2}$ & 2.0204 & 2.0216 & 2.0209 \\
     \hline
     $\|\cdot\|_{L_\infty}$ & 1.6468 & 1.6487 & 1.6511 \\
     \hline
     $\|\cdot\|_{H_1}$ & 1.3030 & 1.3043 & 1.3081 \\
     \hline
    \end{tabular}
\end{table}

\subsection{Main numerical results}
Let $f_m=0$, $\vec{f}_s=0$, $ \psi_0 = 0$, on $\Gamma_m$, $\vec{u}_s = (1,0)^T$, on $\Gamma_{s_1}$, $\vec{u}_s = (0,0)^T$, on $\Gamma_{s_2}$, and $\vec{u}_s = (1,0)^T$, on $\Gamma_{s_3}$.

For exhibiting the stochastic property of our problem, four samples of numerical solutions on the mesh $h_{L} = 1/32$ with four different samples of $\mathbb{K}(\omega,x)$ are shown in the Figure \ref{four_s}.

\begin{figure}[H]
    \centering
    \includegraphics[width=6in]{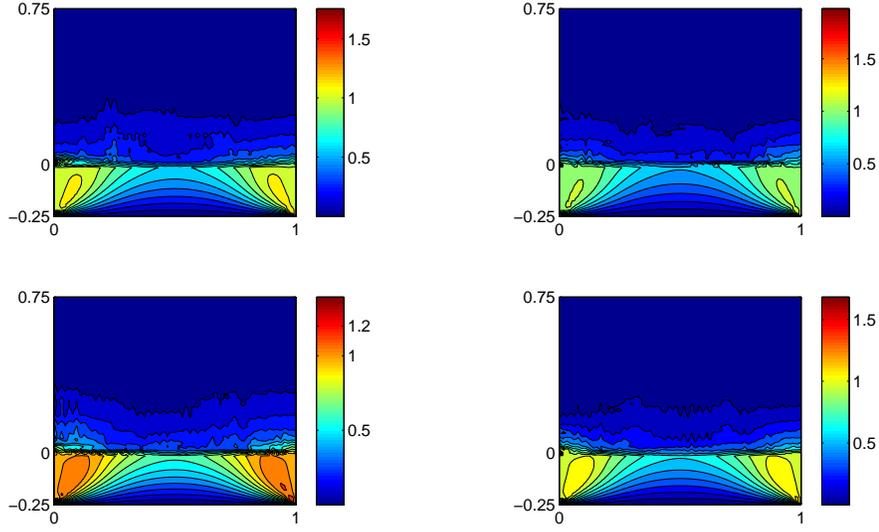}
    \caption{Four samples of solution at $h=1/32$, color represents the speed of flow.}
    \label{four_s}
\end{figure}

For the hierarchical quasi-uniform triangulation mesh $\{\mathcal{T}_{\ell}\}_{\ell=0}^{L}$, four levels are chosen, i.e., $h_\ell=\frac{2^{-\ell}}{4}, \ell=0,1,2,3$ with $h_0=1/4$. An explicit numerical method is needed to determine the parameter $\gamma$ in calculating $\{N^{ML}_{\ell}\}_{\ell=0}^3$, which is needed for MGMLMC method. Based on $C_{\ell} = \mathcal{O}(h_{\ell}^{-\gamma}), \ell = 0,1,2,3$, we can compute $\gamma$ after the computational cost $\{C_{\ell}\}_{\ell=0}^{3}$ of a few samples at every level are recorded. The cpu time and tic-toc time with different mesh size are shown in the Table \ref{time_vs_h}, and the corresponding $\gamma$ are 2.0536, 2.4549, which are illustrated in the Figure \ref{gamma_phi_D}. In this paper, we choose $\gamma = 2.4549$.

\begin{table}[H]
    \centering
    \caption{Computational cost with different mesh size}
    \label{time_vs_h}
    \begin{tabular}{c|cccc}
    \hline
     h & $1/4$ & $1/8$ & $1/16$ & $1/32$ \\
     \hline
     cpu time $\ \left( \text{sec.} \right)$ & 0.48 & 1.84 & 9.64 & 245.65 \\

     tic-toc time $\ \left( \text{sec.} \right)$ & 0.44 & 1.43 & 8.38 & 240.46 \\
     \hline
    \end{tabular}
\end{table}

\begin{figure}[H]
    \centering
    \includegraphics[width=6in]{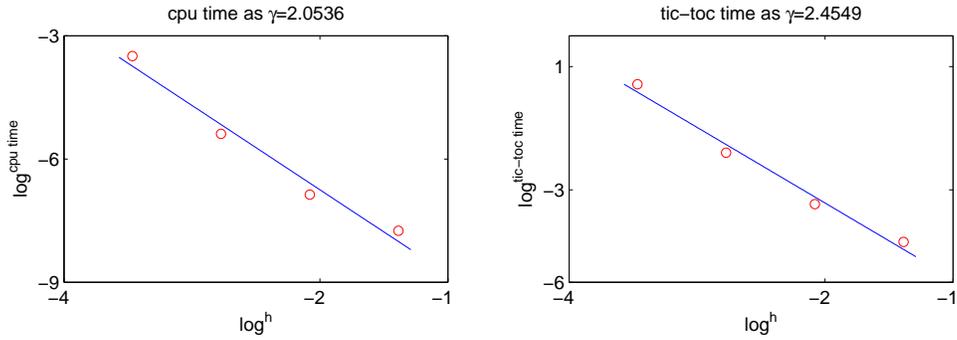}
    \caption{$\gamma$ of cpu time and tic-toc time.}
    \label{gamma_phi_D}
\end{figure}

If the variance $v_0$ at the first level is known, the variance $\{ v_{\ell} \}_{\ell = 0}^{L}$ at the every level could be calculated by $v_{\ell}=\mathcal{O}(h_{\ell}^{\beta})$, while the parameter $\beta$ is approximated by the $\beta$ of stochastic Darcy problem. The variance $v_0$ at the first level is easy to be calculated with low computational cost. Then using the formula \eqref{number_of_samples_at_every_level} with the parameters $\beta$ and $\gamma$ we have gained, we can obtain the number of samples at every level based on the optimization problem \eqref{optimization_N_as_giveb_error} with the given sampling error $e_L$. The numbers of samples on every level with given sampling error $e_L$ are shown in the Table \eqref{number_of_samples_vs_h}.

\begin{table}[H]
    \centering
    \caption{Number of samples at every level}
    \label{number_of_samples_vs_h}
    \begin{tabular}{c|cc|cccc|c}
     \hline
         \  & $e_3^{SL}$ & $e_3^{ML}$ & $N_0^{ML}$ & $N_1^{ML}$ & $N_2^{ML}$ & $N_3^{ML}$ & $N_3^{SL}$ \\
         \hline
         $\|\cdot\|_{L^2}$ & $2\times 10^{-7}$ & $1.95\times 10^{-7}$ & 2127 & 504 & 83 & 14 & 122 \\
         \hline
         $\|\cdot\|_{L^\infty}$ & $2\times 10^{-6}$ & $1.98\times 10^{-6}$ & 2602 & 701 & 131 & 24 & 139 \\
         \hline
         $\|\cdot\|_{H^1}$ & $3\times 10^{-6}$ & $2.97\times 10^{-6}$ & 3521 & 1071 & 225 & 47 & 146 \\
         \hline
    \end{tabular}
\end{table}

To verify the accuracy of MGMLMC method, the relative errors between solutions of SLMC method and those of MGMLMC method are shown in the Table \ref{RE_of_SLMC_MGMLMC}, and the numerical approximations of expectation of speed on mesh $h_{L}=1/32$ are compared with those two methods in the Figure \ref{beta_SL_MGML}. To illustrate the efficiency of MLMC method and MGMLMC method, the computational cost of SLMC method, MLMC method and MGMLMC method are compared in the Table \ref{Efficiency_of_ML_MG}. Based on these results, it is easy to see that the MGMLMC method significantly reduce the computational cost with the same accuracy as SLMC method.

\begin{table}[H]
    \centering
    \caption{Relative errors of solutions by SLMC and MGMLMC method}
    \label{RE_of_SLMC_MGMLMC}
    \begin{tabular}{c|c|c|c|c|c|c}
    \hline
     \  & $\phi_m$ & $p_s$ & $u_m^1$ & $u_m^2$ & $u_s^1$ & $u_s^2$ \\
     \hline
     $\beta=2.02$ & $3.39\%$ & $0.02\%$ & $0.02\%$ & $0.02\%$ & $4.31\%$  & $3.24\%$\\
     \hline
     $\beta=1.65$ & $3.08\%$ & $0.02\%$ & $0.02\%$ & $0.02\%$ & $3.75\%$  & $2.78\%$\\
     \hline
     $\beta=1.30$ & $2.5\%$ & $0.01\%$ & $0.016\%$ & $0.015\%$ & $2.65\%$  & $2.07\%$\\
     \hline
    \end{tabular}
\end{table}

\begin{figure}[H]
    \centering
    \includegraphics[width=5.5in]{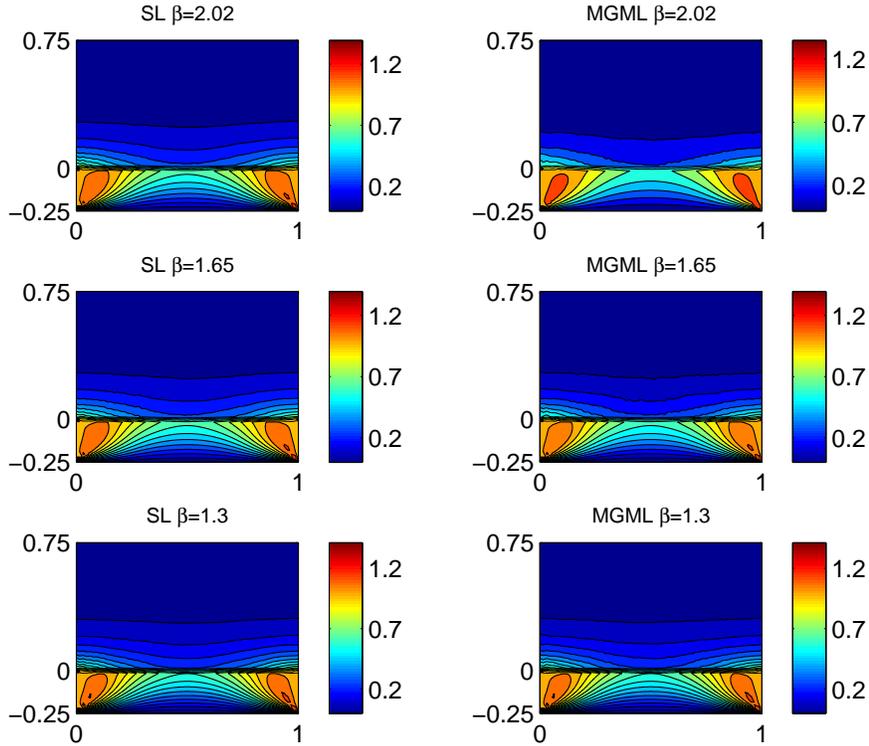}
    \caption{Left: Numerical expectation of speed by SLMC; Right: Numerical \\ expectation of speed by MGML. Color represents the speed of flow.}
    \label{beta_SL_MGML}
\end{figure}

\begin{table}[H]
    \centering
    \caption{Efficiency of MLMC and MGMLMC}
    \label{Efficiency_of_ML_MG}
    \begin{tabular}{c|ccc|c|c}
    \hline
    \  & $T_c^{SL}$ (sec.) & $T_c^{ML}$ (sec.) & $T_c^{MGML}$ (sec.) & $T_c^{ML}/T_c^{SL}$ & $T_c^{MGML}/T_c^{SL}$ \\
     \hline
     $\|\cdot\|_{L^\infty}$ & 30791 & 7315 & 1896 & $23.76\%$ & $6.16\%$  \\
     \hline
     $\|\cdot\|_{L^2}$ & 27025 & 4436 & 1226 & $16.41\%$ &  $4.54\%$  \\
     \hline
     $\|\cdot\|_{H^1}$ & 32342 & 13701 & 3324 & $42.36\%$ & $10.28\%$ \\
     \hline
    \end{tabular}
\end{table}

\section{Conclusion}
In this paper, for the stochastic Stokes-Darcy interface problem, we proved the well-posedness of weak solution, and developed an accurate and efficient multi-grid multi-level Monte Carlo method to solve the numerical approximations. In the proof of the well-posedness, we overcame the difficulties caused by the random hydraulic conductivity both in the porous media domain and on the interface. For the MLMC method, we provided a strategy to calculate the number of samples on every level. We verified the features of the numerical method and the theoretical conclusions.


\end{document}